\documentclass[a4paper, 11pt, final,underruledhead]{paper}
\usepackage{amssymb}
\usepackage{amsmath}
\usepackage{amsthm}
\usepackage{math}
\usepackage[mathscr]{euscript}
\usepackage{enumerate}
\usepackage[all]{xy}
\usepackage{graphics}
\usepackage[dvips]{graphicx}

\CompileMatrices

\newif\ifhozna \hoznafalse

\newif\ifcdn\cdntrue

\begin{document}



\let\goth\mathfrak

\def\myend{{}\hfill{\small$\bigcirc$}}

\newtheorem{reprx}[thm]{Representation}
\newenvironment{repr}{\begin{reprx}\normalfont}{\myend\end{reprx}}
\newtheorem{cnstrx}[thm]{Construction}
\newenvironment{constr}{\begin{cnstrx}\normalfont}{\myend\end{cnstrx}}
\def\classifname{Classification}
\newtheorem{classification}[thm]{\classifname}
\newenvironment{classif}{\begin{classification}\normalfont}{\myend\end{classification}}
\newcounter{facto}
\newenvironment{facto}[1]{%
\refstepcounter{facto}\begin{fact*}\space{\bf {#1}.\thefacto.}\space}{%
\end{fact*}}

\def\myend{{}\hfill{\small$\bigcirc$}}

\newenvironment{ctext}{%
  \par
  \smallskip
  \centering
}{%
 \par
 \smallskip
 \csname @endpetrue\endcsname
}

\newcounter{typek}\setcounter{typek}{0}
\def\thetypek{\roman{typek}}
\def\typitem#1{\refstepcounter{typek}\item[\normalfont(\roman{typek};\hspace{1ex}{#1})\quad]}
\newenvironment{typcription}{\begin{description}\itemsep-2pt\setcounter{typek}{0}}{%
\end{description}}
\newcounter{ostitem}\setcounter{ostitem}{0}

\def\id{\mathrm{id}}

\newcommand{\msub}{\mbox{\large$\goth y$}}    
\def\suport{{\mathrm{supp}}}
\newcommand{\nat}{{N}}   
\def\bagnom#1#2{\genfrac{[}{]}{0pt}{}{#1}{#2}}

\def\kros(#1,#2){c_{\{ #1,#2 \}}}
\def\skros(#1,#2){{\{ #1,#2 \}}}
\def\LineOn(#1,#2){\overline{{#1},{#2}\rule{0em}{1,5ex}}}
\def\inc{\mathrel{\strut\rule{3pt}{0pt}\rule{1pt}{9pt}\rule{3pt}{0pt}\strut}}
\def\lines{{\cal L}}
\def\linessub{{\cal G}}
\def\kliki{{\cal K}}
\def\rozby{{\mathscr G}}
\def\tran{{\mathscr T}}
\def\collin{\sim}
\def\wspolin{{\bf L}}
\def\ncollin{\not\collin}

\def\VerSpace(#1,#2){{\bf V}_{{#2}}({#1})}
\def\GrasSpace(#1,#2){{\bf G}_{{#2}}({#1})}
\def\GrasSpacex(#1,#2){{\bf G}^\ast_{{#2}}({#1})}
\def\VeblSpSymb{{\bf P}\mkern-10mu{\bf B}}
\def\VeblSpace(#1){\VeblSpSymb({#1})}
\def\xwpis#1#2{{{\triangleright}\mkern-4mu{{\strut}_{{#1}}^{{#2}}}}}
\def\MultVeblSymb{{\sf M}\mkern-10mu{\sf V}}        
\def\MultVeblSymq{{\sf T}\mkern-6mu{\sf P}}        
\def\xwlep(#1,#2,#3,#4){{\MultVeblSymb^{#1}{\xwpis{#2}{#3}}{#4}}}
\def\xwlepp(#1,#2,#3,#4,#5){{\MultVeblSymb_{#1}^{#2}{\xwpis{#3}{#4}}{#5}}}
\def\xwlepq(#1,#2,#3,#4,#5){{\MultVeblSymq_{#1}^{#2}{\xwpis{#3}{#4}}{#5}}}
\def\konftyp(#1,#2,#3,#4){\left( {#1}_{#2}\, {#3}_{#4} \right)}
\newcount\liczbaa
\newcount\liczbab
\def\binkonf(#1,#2){\liczbaa=#2 \liczbab=#2 \advance\liczbab by -2
\def\doa{\ifnum\liczbaa = 0\relax \else
\ifnum\liczbaa < 0 \the\liczbaa \else +\the\liczbaa\fi\fi}
\def\dob{\ifnum\liczbab = 0\relax \else
\ifnum\liczbab < 0 \the\liczbab \else +\the\liczbab\fi\fi}
\konftyp(\binom{#1\doa}{2},#1\dob,\binom{#1\doa}{3},3) }
\let\binconf\binkonf

\def\perspace(#1,#2){\mbox{\boldmath$\Pi$}({#1},{#2})}
\def\vergras{{\goth R}}
\def\starof(#1){{\mathop{\mathrm S}(#1)}}
\def\topof(#1){{\mathop{\mathrm T}(#1)}}


\def\Ver{{\sf\bfseries VER}}
\def\Des{{\sf\bfseries DES}}
\def\Desv{{\sf\bfseries DES'}}
\def\Desr{{\sf\bfseries DES''}}
\def\MTC{{\sf STP}}
\def\MVC{{\sf MVC}}

\def\psts{partial Steiner triple system}
\def\PSTS{{\sf PSTS}}
\def\BSTS{{\sf BSTS}}
\def\asumpt{\rm({\sf P})}
\def\desAx{\rm({\sf Des})}

\def\brak{\relax} 


\ifcdn\bgroup  

\title[Configurations representing a skew perspective]{%
Configurations representing a skew perspective}
\author{Kamil Maszkowski, Ma{\l}gorzata Pra{\.z}mowska, Krzysztof Pra{\.z}mowski}

\maketitle

\begin{abstract}
  A combinatorial object representing schemas of, possibly skew, perspectives,
  called {\em a configuration of skew perspective} is defined.
  Some classifications of skew perspectives are presented.
\newline
{\em key words}: Veblen (Pasch) configuration, (generalized) Desargues configuration,
binomial configuration, complete (free sub)subgraph, perspective.
\\
MSC(2000): 05B30, 51E30.
\end{abstract}

\section*{Introduction}

The term {\em perspective}, the title subject of this paper, 
is used, primarily, in architecture drawings and, after that,
in descriptive and projective geometry.
It refers, in fact, to a (central) projection i.e. to a correspondence between
objects of one space (points, lines, planes, spheres, $\ldots$) and objects of another
space, while both two are subspaces of a third one (``the real world'').
Such a projection is central if the lines which join corresponding points
meet in a ``center'' (an `eye').
In investigations of projective geometry central projections between lines,
planes etc. play a crucial role; e.g. they are used to characterize so called
projective collineations, projective correspondence, and similar notions of projective
geometry
(see standard textbooks like \cite{projmono1}, \cite{projmono2},
textbooks on general geometry like \cite{hilbert}, or more advanced
investigations in \cite{projectiv}, \cite{coxdes}).
Projections are also used to characterize Pasch property (invariance of an order)
in projective and chain geometries (see e.g. \cite{projchain}).
And in many other places.
Roughly speaking, a projection is a local linear collineation.

\medskip
One can talk about perspective in a more general settings of (finite) systems of 
points: of configurations, or even: of graphs.
General requirements that should be met by such a perspective we formulate 
in \asumpt\ in Section \ref{sec:under}.
And the most `instructive' and `vivid' example that one should have in mind is the classical
Desargues configuration considered as a perspective between two triangles
(see e.g. \cite[Ch. III, \S 19]{hilbert}).
This configuration was generalized in many directions, e.g. to take into account 
a perspective between $m$-simplices that may be realized in a projective space 
(see \cite{perspect}, \cite{mveb2proj}, \cite[Generalized Desargues Configuration]{doliwa2}).
As we said, a perspective $\pi$ is a local collineation: while defined primarily
on the points it extends uniquely to a map $\overline{\pi}$ defined on more complex objects
such as lines (planes, chains and so on).
\par
It appears even in the smallest reasonable case of ${10}_{3}$-configurations that 
the Desargues configuration has two cousins, both two realizable in a projective plane
over a field, such that the associated perspectives $\overline{\pi}$ are skew.
Namely, the points in which intersect sides of a triangle and their images under
$\overline{\pi}$ do not colline.
However, another correspondence $\xi$ can be found, $\xi \neq \overline{\pi}$
such that a side $e$ of a triangle and $\xi(e)$ intersect on a fixed line (an `axis').

\medskip
Is it possible to generalize this class of perspectives to `bigger' simplices?
Formally, the answer is trivial: it suffices to introduce
suitable (`constructive') definition (see Construction \ref{def:pers}).
After that, the natural question aries, how to classify the obtained structures,
how to characterize them and their geometry.
Some partial answers to these questions  are given in this paper.
First, we note that, from a general perspective, our perspectives are exactly
the binomial \psts s which freely contain at least two maximal complete graphs (see \cite{klik:binom}).
From this point of view, a complete characterization of 
all the skew perspectives seems far to reach.
Applying the requirement that $\xi$ extends to the line graphs of the respective
simplices (i.e. $\xi$ maps concurrent edges onto concurrent edges)
we arrive much closer to a complete classification of skew perspectives.
In particular, we obtain such a classification for perspectives whose axes are 
generalized Desargues configurations (Prop. \ref{prop:class-grasaxis}).
%
%
In Section \ref{subsec:konter} we compare the obtained perspectives with
some other known $\konftyp(15,4,20,3)$-configurations.
Examples show that 
a variety of quirks may appear,
a configuration in question may be represented in a `regular' way and - with other
centre chosen - as a perspective with quite irregular skew.

\medskip
The question which of our perspectives are simultaneously multiveblen configurations
(another class of {\psts s} generalizing a projective perspective, introduced in \cite{pascvebl})
is discussed in Proposition \ref{prop:pers-mveb}.
The natural question which of so generalized perspectives can be realized in a
Desarguesian projective space is discussed in Section \ref{ssec:pers2proj}
and is completely solved in case of $\konftyp(15,4,20,3)$-skew-perspectives.
Some remarks on configurational axioms associated with our configurations are made
in Section \ref{ssec:axioms}.



\section{Underlying ideas and basic definitions}\label{sec:under}

Let us begin with introducing some, standard, notation.
Let $X$ be an arbitrary set.
The symbol $S_X$ stands for the family of permutations of $X$.
Let $k$ be a positive integer; 
we write $\sub_k(X)$ for the family of $k$-element subsets of $X$.
Then $K_X = \struct{X,\sub_2(X)}$ is the complete graph on $X$;
$K_n$ is $K_X$ for any $X$ with $|X| = n$. Analogously, $S_n = S_X$.
\par\noindent
A {\em $\konftyp(\nu,r,b,\varkappa)$-configuration} is a configuration 
(a {\em partial linear space} i.e. an incidence structure with blocks ({\em lines})
pairwise intersecting in at most a point)
with
$\nu$ points, each of rank $r$, and $b$ lines, each of rank (size) $\varkappa$.
A {\em \psts}\ (in short: a \PSTS) is a partial linear space 
with all the lines of size $3$.
A $\binkonf(n,0)$-configuration is a \psts, it is called 
a {\em binomial \psts}.
\par\noindent
We say that a graph $\cal G$ is {\em freely contained} in a 
configuration $\goth B$ iff 
the vertices of $\cal G$ are points of $\goth B$, 
each edge $e$ of $\cal G$ is contained in a line $\overline{e}$ of $\goth B$,
the above map $e \mapsto \overline{e}$ is an injection, 
and lines of $\goth B$ which contain disjoint edges of $\cal G$ do not 
intersect in $\goth B$.
If $\goth B$ is a $\binkonf(n,0)$-configuration and ${\cal G} = K_X$
then $|X| + 1 \leq n$. Consequently, $K_{n-1}$ is a maximal complete graph
freely contained in a binomial $\binkonf(n,0)$-configuration.
Further details of this theory are presented in \cite{klik:binom},
relevant results will be quoted in the text, when needed.

\bigskip
In the paper we aim to develop a theory of configurations which characterize abstract
properties of a perspective between two graphs.
Let us start with the following general (evidently: unprecise yet) requirements.
%
\begin{quotation}\em\noindent
  When we talk about a {\em perspective} between two graphs 
  ${\goth G}_1$ and ${\goth G}_2$, where 
  $X_i$ is the the set of vertices  
  and ${\cal E}_i$ is the set of edges of ${\goth G}_i$
  then we have
  \begin{itemize}\itemsep-2pt\def\labelitemi{--}
    \item
      a {\em perspective center}: a point $p$ such that {\em perspective rays},
      lines through $p$ establish a one-to-one correspondence 
      $\pi$ ({\em point perspective}) between $X_1$ and $X_2$,
      and
    \item
      an {\em axis}: a configuration (disjoint with $X_1\cup X_2$) such that
      a one-to-one correspondence 
      $\xi$ ({\em line perspective}) between ${\cal E}_1$ and ${\cal E}_2$
      is characterized by the condition:
      \begin{quotation}\noindent
        an edge in ${\cal E}_1$ and its counterpart in ${\cal E}_2$, suitably extended,
        intersect on the axis.
      \end{quotation}
  \end{itemize}
(comp. \cite[Prop. 2.6]{klik:binom}, 
\cite[Repr. 2.4, Repr. 2.5]{skewgras} or standard textbooks on projective geometry, e.g.
\cite{projmono1},\cite{projmono2}).
\quad\quad\strut\hfill\asumpt 
\end{quotation}
The associated configuration consists of the points in $X_1\cup X_2$ completed by 
the center and the intersections of extended edges, and the minimal amount of the lines
which join these intersection points.

This approach is, however, too general.
We want our perspective to yield a {\em regular configuration} i.e. a one with 
all the points of the same rank.
It is seen that the size of the lines must be $3$.
The rank of the perspective center is $n = |X_1| = |X_2|$, therefore the rank of 
$a\in X_1$ in ${\goth G}_1$ must be $n-1$ and therefore ${\goth G}_1$ and ${\goth G}_2$
both are complete $K_{n}$-graphs.
So, unhappily, only perspectives between complete graphs can be characterized
in accordance with our requirements \asumpt.
On the other hand, this restriction leads us to a quite nice part of the theory 
of configurations.

\medskip
So, let us pass to a more exact formulation of  requirements \asumpt. 
\begin{constr}\label{def:pers}
  Let $I$ be a nonempty finite set, $n := |I|$.
  In most parts, without loss of generality, we assume that
  $I = I_n =  \{ 1,\ldots,n \}$.
  Let 
    $A = \{a_i\colon i\in I \}$ and $B = \{b_i\colon i\in I\}$
  be two disjoint $n$-element sets, let $p\notin A\cup B$.
\newline
  Then we take a $\binom{n}{2}$-element set
  $C = \{c_u\colon u\in\sub_2(I)\}$
  disjoint with $A\cup B\cup\{p\}$.
  Set 
\begin{equation*}
  {\cal P} = A\cup B\cup\{p\}\cup C.
\end{equation*}
  Let us fix a permutation $\sigma$ of $\sub_2(I)$ and write
\begin{eqnarray*}
  {\cal L}_p & := & \big\{ \{ p,a_i,b_i \}\colon i\in I \big\},
  \\
  {\cal L}_A & := & \big\{ \{ a_i,a_j,c_{\{ i,j \}} \}\colon \{i,j\}\in\sub_2(I) \big\},
  \\
  {\cal L}_B & := & \big\{ \{ b_i,b_j,c_{\sigma^{-1}(\{ i,j \})} \} \colon \{i,j\}\in\sub_2(I) \big\}.
\end{eqnarray*}
  Finally, let ${\cal L}_C$ be a family of $3$-subsets of $C$ such that
  ${\goth N} = \struct{C,{\cal L}_C}$ is a 
  $\binconf(n,0)$-configuration.
  Set
$$
  {\cal L} = {\cal L}_p \cup {\cal L}_A\cup {\cal L}_B\cup {\cal L}_C
  \text{ and }
  \perspace(n,\sigma,{\goth N}) := \struct{{\cal P},{\cal L}}.
$$
  The structure $\perspace(n,\sigma,{\goth N})$ will be referred to as a 
  {\em skew perspective} with the {\em skew} $\sigma$.
\end{constr}
We frequently shorten $c_{\{i,j\}}$ to $c_{i,j}$.
In many cases, the parameter $\goth N$ will not be essential and then it will be omitted,
we shall write simply $\perspace(n,\sigma)$.
In essence, the names "$a_i$", "$c_{i,j}$" are -- from the point of view of 
mathematics -- arbitrary, and could be replaced by any other labelling
(cf. analogous problem of labelling in \cite[Constr. 3, Repr. 3]{pascvebl}
or in \cite[Rem 2.11, Rem 2,13]{STP3K5}, \cite[Exmpl. 2]{pascvebl}).
Formally, one can define $J = I \cup \{a,b\}$,
$x_i = \{x,i\}$ for $x\in\{a,b\} =: p$ and $i\in I$, and $c_u = u$ for $u\in \sub_2(I)$.
After this identification $\perspace(n,\sigma)$ becomes a structure defined
on $\sub_2(J)$.
Then, it is easily seen that
\begin{equation}\label{paramy}
  \perspace(n,\sigma,{\goth N}) \text{ is a }\binconf(n,+2)\text{ configuration}.
\end{equation}
In particular, it is a \psts\ (a partial linear space),
so we can use standard notation:
$\LineOn(x,y)$ stands for the line which joins two collinear points $x,y\in{\cal P}$,
and then we define the partial operation $\oplus$ with the following requirements:
$x\oplus x = x$, $\{x,y,x\oplus y\}\in{\cal L}$ whenever $\LineOn(x,y)$ exists.
%
Observe then that 
(cf. \cite[Eq. (1), the definition of 
{\em combinatorial Grassmannian} $\GrasSpace(n,2)$]{perspect})
\begin{equation}\label{gras:pers}
  \GrasSpace(n+2,2) = \GrasSpace(J,2) = \struct{\sub_2(J),\sub_3(J),\subset}
  \cong \perspace(n,\id_{I_n},{\GrasSpace(I_n,2)}).
\end{equation}

It is clear that $A^\ast = A \cup \{p\}$ and $B^\ast = B\cup\{p\}$ are two $K_{n+1}$-graphs
freely contained in $\perspace(n,\sigma)$.
Applying the results \cite[Prop. 2.6 and Thm. 2.12]{klik:binom} 
we immediately obtain the following.
\begin{fact}
  Let $N = n+2$. The following conditions are equivalent.
  \begin{sentences}\itemsep-2pt
  \item
    $\goth M$ is a binomial $\binkonf(N,0)$-configuration which freely contains 
    two $K_{N-1}$-graphs.
  \item
    ${\goth M}\cong \perspace(n,\sigma,{\goth N})$ for a $\sigma\in S_{\sub_2(I)}$
    and a $\binkonf(n,0)$-configuration $\goth N$ defined on $\sub_2(I)$.
  \end{sentences}
\end{fact}
Consequently, the configurations defined by \ref{def:pers} 
are essentially known, 
but no {\em general} classification of them is known, though.

The map
$$
  \pi = \big( a_i \longmapsto b_i,\; i\in I \big)
$$
is a point-perspective of $K_A$ onto $K_B$ with center $p$.
Moreover, the map
$$
  \xi = \big( \LineOn(a_i,a_j) \longmapsto \LineOn(b_{i'},b_{j'}), \;
  \sigma( \{ i,j \} ) = \{i',j'\} \in \sub_2(I) \big)
$$
is a line perspective, where $\goth N$ is the axial configuration of our perspective.
Consequently, $\perspace(n,\sigma,{\goth N})$ satisfies the requirement \asumpt\ i.e.
it is a schema of a perspective of some type.
Contrary to the approach of \cite{klik:binom},
following the approach of this paper we can 
better analyze some particular properties of the perspective $(\pi,\xi)$.
\begin{lem}\label{lem:meetinvar}
  The map $\xi$ maps intersecting edges of $K_A$ onto intersecting edges of $K_B$
  iff either
  \begin{sentences}
    \item\label{lem1:cas1}
      there is a permutation $\sigma_0\in S_I$ such that 
      \begin{equation}\label{2perm}
          \sigma(\{ i,j \}) = \overline{\sigma_0}(\{ i,j \})
	   =\{ \sigma_0(i),\sigma_0(j) \}
      \end{equation}
      for every $\{i,j\}\in\sub_2(I)$, or
    \item\label{lem1:cas2}
      $n = 4$ and $\sigma(u) = I\setminus \overline{\sigma_0}(u)$
      for every $u\in\sub_2(I)$, where $\overline{\sigma_0}$ is defined by \eqref{2perm}
      for some $\sigma_0\in S_I$.
  \end{sentences}
  In case \eqref{lem1:cas1}, $\xi$ preserves the (ternary) concurrency of edges,
  and in case \eqref{lem1:cas2}, the concurrency is not preserved.
\end{lem}
\begin{proof}
  One can identify an edge $\{ a_i,a_j \}$ of $K_A$ with $\{ i,j \}\in\sub_2(I)$; 
  analogously we identify 
  $\sub_2(B)\ni\{ b_i,b_j \} \mapsto \{ i,j \}\in\sub_2(I)$.
  After this identification $\xi\in S_{\sub_2(I)}$, and $\xi$ preserves the edge-intersection
  iff it preserves set-intersection. The claim is just a reformulation of the folklore
  (cf. \cite{klin}, \cite[Prop. 1.5]{perspect}, \cite[Prop. 15]{pascvebl}).
\end{proof}
A more detailed analysis of
the case \ref{lem:meetinvar}\eqref{lem1:cas2} is addressed to another paper.
%
\begin{note}\normalfont
  If $\sigma_0\in S_I$ we frequently identify $\sigma_0$, $\overline{\sigma_0}$,
  and the corresponding map $\xi$.
  Consequently, if $\sigma\in S_I$ we write
  $\perspace(n,\sigma,{\goth N})$ in place of
  $\perspace(n,\overline{\sigma},{\goth N})$.
\end{note}

\begin{prop}\label{prop:iso0}
  Let $f \in S_{{\cal P}}$, $f(p) = p$, $\sigma_1,\sigma_2 \in S_{\sub_2(I)}$,
  and 
  ${\goth N}_1, {\goth N}_2$ be two $\binkonf(n,0)$- configurations defined
  on $\sub_2(I)$.
  The following conditions are equivalent.
  \begin{sentences}\itemsep-2pt
   \item\label{propiso0:war1}
     $f$ is an isomorphism 
     of $\perspace(n,\sigma_1,{\goth N}_1)$ onto $\perspace(n,\sigma_2,{\goth N}_2)$.
   \item\label{propiso0:war2}
     There is $\varphi \in S_I$ such that
     one of the following holds
     \begin{eqnarray} 
     \label{iso0:war1-1}
        \overline{\varphi} \text{ ( comp. \eqref{2perm})} & {\text is } & \text{ an isomorphism of } 
            {\goth N}_1 \text{ onto } {\goth N}_2,
     \\ 
     \label{propiso0:typ1}
       f(x_i) = x_{\varphi(i)},\; x = a,b,\; &&
       f(c_{\{i,j\}}) = c_{\{\varphi(i),\varphi(j)\}},
       \quad i,j\in I, i\neq j,
     \\ \label{iso0:war2}
       \overline{\varphi} \circ \sigma_1 & = & \sigma_2 \circ \overline{\varphi},
     \end{eqnarray}
     or
     \begin{eqnarray} 
     \label{iso0:war1-2}
        \sigma_2^{-1}\overline{\varphi} & \text{  is } & \text{ an isomorphism of } 
            {\goth N}_1 \text{ onto } {\goth N}_2,
     \\ 
     \label{propiso0:typ2}
       f(a_i) = b_{\varphi(i)},\; f(b_i) = a_{\varphi(i)}, &&
        f(c_{\{i,j\}}) = c_{\sigma_2^{-1}\{\varphi(i),\varphi(j)\}},
       \; i,j\in I, i\neq j,
     \\ \label{iso0:war3}
       \overline{\varphi} \circ \sigma_1 & = & \sigma_2^{-1} \circ \overline{\varphi}.
     \end{eqnarray}
  \end{sentences}
\end{prop}
\begin{proof}
  Write ${\goth M}_l = \perspace(n,\sigma_l,{\goth N}_l)$ for $l=1,2$.
  \par
  Assume \eqref{propiso0:war1}. Since 
  exactly two free $K_{n+1}$ subgraphs of ${\goth M}_l$ ($l=1,2$) pass through $p$
  (cf. \cite[Prop.'s 2.6, 2.7]{klik:binom}),
  one of the following holds
  \begin{enumerate}[(a)]\itemsep-2pt
  \item\label{war0:1} $f(A) =A$ and $f(B)= B$, or
  \item\label{war0:2} $f(A) = B$ and $f(B) = A$.
  \end{enumerate}
  Assume, first, \eqref{war0:1}. Consequently, there is a permutation $\varphi\in S_I$
  such that $f(a_i) = a_{\varphi(i)}$ for each $i\in I$.
  This yields $f(b_i) = f(p) \oplus f(a_i) = b_{\varphi(i)}$, and, finally
  $f(c_{i,j}) = f(a_i\oplus a_j) = \ldots = c_{\varphi(i),\varphi(j)}$.
  This justifies \eqref{propiso0:typ1}. 
  Since $f$ preserves the lines of $\goth N$,
  from \eqref{propiso0:typ1} we infer \eqref{iso0:war1-1}.
  Finally, the equation
  \begin{math}
    c_{ \overline{\varphi}(\sigma_1^{-1}(\{ i,j \}))} 
    = f( c_{ \sigma_1^{-1}(\{ i,j \})} ) 
    = f (b_i \oplus b_j)
    = f(b_i) \oplus f(b_j)
    = b_{\varphi(i)} \oplus b_{\varphi(j)}
    = c_{ \sigma_2^{-1}(\{ \varphi(i) , \varphi(j)\} ) }
   \end{math}
   justifies \eqref{iso0:war2}.
   \par
   In case \eqref{war0:2} the reasoning goes analogously. We only need to note that 
   \begin{math}
     f(c_{\{i,j\}}) = f(b_{\varphi(i)} \oplus b_{\varphi(j)}) = c_{\sigma_2^{-1}\overline{\varphi}(\{ i,j \})}, 
   \end{math}
   which justifies the last condition in \eqref{propiso0:typ2} and yields \eqref{iso0:war1-2}. 
   \par\medskip
   Conversely, if \eqref{propiso0:war2} is assumed we directly verify that 
   $f(x\oplus y) = f(x) \oplus f(y)$ holds for all $x,y\in (A \cup B)$, which 
   proves \eqref{propiso0:war1}.
\end{proof}

\begin{lem}\label{lem:nextgraf00}
    Assume that $\perspace(n,\sigma,{\goth N})$ freely contains a complete $K_{n+1}$-graph $G\neq K_{A^\ast},K_{B^\ast}$,
    $\sigma\in S_{\sub_2(I)}$.
    Then there is $i_0 \in I$ such that 
    $\starof(i_0) = \{ c_u\colon i_0\in u\in\sub_2(I) \}$
    is a collinearity clique in $\goth N$ freely contained in it.
    Moreover, 
    \begin{equation}\label{wzor:extragraf00}
    G = G_{(i_0)} \; := \; \{ a_{i_0},b_{i_0} \} \cup \starof(i_0).
  \end{equation}
\end{lem}
\begin{proof}
  Let $G\neq K_{A^\ast},K_{B^\ast}$ be a complete $K_{n+1}$-graph freely contained in 
  $\perspace(n,\sigma,{\goth N}) =: {\goth M}$.
  Then $p$, $G\cap A$, and $G\cap B$ form a triple of collinear points 
  (cf. \cite[Prop. 2.7]{klik:binom}).
  So, there is $i_0\in I$ such that $a_{i_0},b_{i_0} \in G$. And 
  $G\setminus \{a_{i_0},b_{i_0}\} \subset C$. The set of points in $C$ which are collinear
  with $a_{i_0}$ is exactly $\starof(i_0)$; it contains $G$ and its cardinality
  is $n-1$, and therefore $G = G_{i_0}$.  
  Since $G$ is a clique, we conclude with:
  $\starof(i_0)$ is a clique in $\goth N$.
  Clearly, it is freely contained in $\goth N$.
\end{proof}

\begin{exm}\label{exm:7}
  Let us define $\zeta\colon\sub_2(I_4)\longrightarrow\sub_2(I_4)$ by the following formula:
  \begin{equation}
    \zeta(\{ u \}) = \left\{  
    \begin{array}{ll} \{ u \} & \text{when } u \neq \{1,2\}, \{3,4\},
    \\ 
    I_4\setminus u & \text{when } u \in \{ \{1,2\}, \{3,4\} \}.
    \end{array}\right.
  \end{equation}
  Note that $\zeta^{-1} = \zeta$.
  \par
  Clearly, $\zeta$ does not preserve edge-intersection.
  It is easy to verify that ${\goth M} = \perspace(4,\zeta,{\GrasSpace(I_4,2)})$ has no free $K_5$-subgraph
  distinct from $A^\ast$ and $B^\ast$. 
  Any isomorphism of $\goth M$ onto ${\goth M}_0 = \perspace(4,\zeta_0,{\goth N})$
  ($\zeta_0 = \overline{\sigma_0}$ or $\zeta_0 = \varkappa\overline{\sigma_0}$, $\sigma_0\in S_{I_4}$,
  $\varkappa(u) = I_4 \setminus u$, notation of \ref{lem:meetinvar})
  maps $p$ onto $p$, so it determines (use \ref{prop:iso0}) a permutation $\varphi\in S_{I_4}$
  such that $\zeta = \zeta_0^{\overline{\varphi}}$.
   Since no such $\zeta_0,\varphi$ exist, 
   {\em there is no skew perspective that preserves edge intersection and is isomorphic to $\goth M$.} 
\end{exm}


\section{Perspectivities associated with permutations of indices: %
general properties}\label{sec:perm-pers}

\begin{note}\normalfont
  Let ${\goth M} = \perspace(n,\sigma,{\goth N})$ be a skew perspective with $\sigma\in S_{I_4}$.
  If $n = 1$ then $\goth M$ is a single line.
  If $n = 2$ then $\goth N$ is a single point and $\sigma = \id_{\sub_2(I_2)}$,
  and then $\goth M$ is the Veblen configuration $\GrasSpace(I_4,2)$
  (the configuration in question is also frequently called the Pasch
  configuration, cf. e.g. \cite{pasz}).
  If $n = 3$ then ${\goth N}$ is a single $3$-line.
  The configurations $\perspace(3,\sigma)$ were determined and characterized in
  \cite{klik:VC};
  these are exactly
  \begin{itemize}\itemsep-2pt
  \item
    the {\em Desargues configuration} $\perspace(3,\id_{I_3})$,
  \item
    the {\em fez configuration} $\perspace(3,{(1,2,3)})$, and
  \item
    the {\em Kantor configuration} $\perspace(3,{(1)(2,3)})$;
  \end{itemize}
  cf. \cite[Repr. 2.6]{klik:VC}
\end{note}

In this section we consider structures $\perspace(n,\sigma)$ where
$\sigma\in S_n$ and $n > 3$.
Two very useful formulas will be frequently used without explicit quotation:
\begin{ctext}
  $a_i \oplus a_j = c_{i,j}$, and
  $b_i \oplus b_j = c_{\sigma^{-1}(i),\sigma^{-1}(j)}$
  for each $\{ i,j \}\in\sub_2(I)$
\end{ctext}
so, $\LineOn(a_i,a_j)$ crosses $\LineOn(b_{\sigma(i)},b_{\sigma(j)})$
in $c_{i,j}$.

\begin{lem}\label{lem:nextgraf}
  The following conditions are equivalent.
  \begin{sentences}\itemsep-2pt
  \item\label{lem2:cas1}
    $\perspace(n,\sigma,{\goth N})$ freely contains a complete $K_{n+1}$-graph $G\neq K_{A^\ast},K_{B^\ast}$.
  \item\label{lem2:cas2}
    There is $i_0 \in \Fix(\sigma)$ such that 
    $\starof(i_0) = \{ c_u\colon i_0\in u\in\sub_2(I) \}$
    is a collinearity clique in $\goth N$ freely contained in it.
  \end{sentences}
  In case \eqref{lem2:cas2},
  \begin{equation}\label{wzor:extragraf}
    G_{(i_0)} \; := \; \{ a_{i_0},b_{i_0} \} \cup \starof(i_0)
  \end{equation}
  is a complete graph freely contained in $\perspace(n,\sigma,{\goth N})$.
\end{lem}
\begin{proof}
  Assume \eqref{lem2:cas1}.
  From \ref{lem:nextgraf00}, $G$ has form \eqref{wzor:extragraf00}, so
  $b_{i_0}$ must be collinear with each point in $\starof(i_0)$.
  In other words, for each $j\in I\setminus\{i_0\}$ there is $j'$ such that
  $a_{i_0}\oplus a_j  = c_{i_0,j} = b_{\sigma(i_0)}\oplus b_{\sigma(j)}
  = b_{i_0} \oplus b_{j'}$.
  From this we infer 
  $\{ \sigma(i_0),\sigma(j) \} = \{ i_0,j' \}$, and thus $\sigma(i_0) = i_0$. 
  So, from \eqref{lem2:cas1} we have arrived to \eqref{lem2:cas2}.
  \par
  It is a trivial task to prove that under assumptions \eqref{lem2:cas2} the set defined
  by \eqref{wzor:extragraf} is a required $K_{n+1}$-graph, which proves \eqref{lem2:cas1}.
\end{proof}

Let us note, as a particular case of \ref{prop:iso0}, the following characterization.
\begin{prop}\label{prop:iso1}
  Let $f \in S_{{\cal P}}$, $f(p) = p$, $\sigma_1,\sigma_2 \in S_I$,
  and 
  ${\goth N}_1, {\goth N}_2$ be two $\binkonf(n,0)$- configurations defined
  on $\sub_2(I)$.
  The following conditions are equivalent.
  \begin{sentences}\itemsep-2pt
   \item\label{propiso1:war1}
     $f$ is an isomorphism 
     of $\perspace(n,\sigma_1,{\goth N}_1)$ onto $\perspace(n,\sigma_2,{\goth N}_2)$.
   \item\label{propiso1:war2}
     There is $\varphi \in S_I$ such that
     \begin{equation} \label{iso:war1}
       \overline{\varphi} \text{ ( comp. \eqref{2perm}) is } \text{ an isomorphism of } 
       {\goth N}_1 \text{ onto } {\goth N}_2,
     \end{equation}
     and one of the following holds
     \begin{eqnarray} \label{propiso1:typ1}
       f(x_i) = x_{\varphi(i)},\; x = a,b,\; &&
       f(c_{\{i,j\}}) = c_{\{\varphi(i),\varphi(j)\}},
       \quad i,j\in I, i\neq j,
     \\ \label{iso:war2}
       \varphi \circ \sigma_1 & = & \sigma_2 \circ \varphi,
     \end{eqnarray}
     or
     \begin{eqnarray} \label{propiso1:typ2}
       f(a_i) = b_{\varphi(i)},\; f(b_i) = a_{\varphi(i)}, &&
        f(c_{\{i,j\}}) = c_{\{\varphi(i),\varphi(j)\}},
       \quad i,j\in I, i\neq j,
     \\ \label{iso:war3}
       \varphi \circ \sigma_1 & = & \sigma_2^{-1} \circ \varphi.
     \end{eqnarray}
  \end{sentences}
\end{prop}
As we know (cf. \cite[Prop. 2.6]{klik:binom}), in case \ref{lem:nextgraf}
there is a permutation of the edges of $K_{A^\ast \setminus \{ a_{i_0} \}}$
such that ${\goth M} \cong \perspace(n,\sigma',{\goth N}')$ for an adequate
configuration ${\goth N}'$:
${\goth M}$ {\em is} a skew perspective of $K_{A^\ast\setminus \{ a_{i_0} \}}$
onto $G_{(i_0)}$.
In the case we frequently say 
``${\goth M}\cong \perspace(n,\sigma',{\goth N}')$ 
and {\em $a_{i_0}$ is the perspective center
in $\perspace(n,\sigma',{\goth N}')$}''.
However, 
$\sigma'$ need not to be determined by a permutation of the vertices (cf. \ref{lem:meetinvar}) 
neither $\sigma$ and $\sigma'$ are necessarily conjugate (cf. \ref{prop:iso1}).

\begin{prop}\label{prop:movecenter}
  Let $\starof(i_0)$ be a clique in $\goth N$ for some $i_0\in\Fix(\sigma)$,
  $\sigma\in S_I$, $|I| = n+1\geq 4$ (cf. \ref{lem:nextgraf}).
  The following conditions are equivalent.
  \begin{sentences}\itemsep-2pt
  \item\label{movecenter1}
    $\perspace(n+1,\sigma,{\goth N}) \cong \perspace(n+1,\sigma',{\goth N}')$,
    for a $\sigma' = \overline{\sigma'_0}$, $\sigma'_0 \in S_{n+1}$\space
    and a suitable configuration ${\goth N}'$,
    where $a_{i_0}$ is the perspective center in $\perspace(n+1,\sigma',{\goth N}')$
    of the graphs $G_{(i_0)}$ and $K_{A^\ast \setminus \{ a_{i_0} \}}$.
  \item\label{movecenter2}
    There is $\tau \in S_{I\setminus \{i_0\}}$ such that
    \begin{equation}\label{war:extraskos}
      c_{\{i_0,\tau(i)\}} \oplus c_{\{ i_0,\tau(j) \}} = c_{\{ i,j \}}
    \end{equation}
    for all $i,j\in I$, $i,j\neq i_0$.
  \end{sentences}
\end{prop}
\begin{proof}
  Assume \eqref{movecenter1}.
  Without loss of generality we can assume that $I = \{0,1,\ldots,n\}$ and $i_0 = 0$.
  So, we relabel the points of $\perspace(n+1,\sigma,{\goth N}) =: {\goth M}$ 
  so as $q = a_0$ becomes a 
  perspective center and 
  $a_i: i=1,\ldots,n+1$ and $d_i: i=1,\ldots,n+1$ will be the complete subgraphs that
  are in the respective perspective. Finally, we take $e_{i,j} = a_i\oplus a_j$
  for $\{ i,j \}\in\sub_2(T)$, $T = \{ 1,\ldots,n+1 \}$.
  So, we obtain
  \begin{multline}
    a_{n+1} = p,\;
    d_{i} = q \oplus a_i = c_{0,i} \text{ for } i\in T, i\neq 0,\;
    d_{n+1} = q\oplus a_{n+1} = b_0,
    \\
    e_{i,j} = c_{0,i} \oplus c_{0,j} \; (\text{computed in } {\goth N})
    \text{ for } i,j\in T, i,j\neq 0,\;
    \\
    e_{i,n+1} = b_i \text{ for } i\in T, i\neq 0.
  \end{multline}
  Let $\tau\in S_T$ be the corresponding skew i.e. assume that
  \begin{equation}
    a_i \oplus a_j = e_{i,j} = d_{\tau(i)} \oplus d_{\tau(j)}
  \end{equation}
  for all $\{ i,j \}\in \sub_2(T)$.
  In particular, this yields for $i\in T$, $i\neq n+1$ the following:
  $    a_i \oplus a_{n+1} = $
  \begin{equation}
    b_i = d_{\tau(i)} \oplus d_{\tau(n+1)}
    = \left\{\begin{array}{ll}
        c_{0,\tau(i)}\oplus c_{0,\tau(n+1)} & or
	\\
	c_{0,\tau(i)}\oplus b_0 & \tau(n+1) = n+1, \tau(i)\neq n
	\\
	b_0 \oplus c_{0,\tau(n+1)} & \tau(i) = n+1,\tau(n+1) = 0
      \end{array}
      \right. .
  \end{equation}
  Since $\goth M$ does not contain any line with exactly one point in $B$ and two points in $C$,
  the first possibility is inconsistent.
  So, we end up with $\tau(n+1) = n+1$ and therefore, $\tau\in S_n$.
  If so, we obtain
  \begin{math}
    c_{i,j} = a_i \oplus a_j = e_{i,j} = d_{\tau(i)} \oplus d_{\tau(j)} 
    = c_{0,\tau(i)} \oplus c_{0,\tau(j)}
  \end{math}
  for distinct $1\leq i,j\leq n$.
  This justifies \eqref{war:extraskos}.
  \par\medskip
  The converse reasoning consists in a simple computation: the reasoning above 
  defines, in fact, a required isomorphism.
  It also defines the configuration ${\goth N}'$:
  the formulas 
  \begin{math}
    e_{i,n+1} \oplus e_{j,n+1} = b_i \oplus b_j = c_{\sigma^{-1}(i),\sigma^{-1}(j)} =
    e_{\sigma^{-1}(i),\sigma^{-1}(j)}
  \end{math}
  for $1\leq i,j\leq n$ and $e_{u}\oplus e_v = e_y$ iff $c_u \oplus c_v = c_y$
  for $u,v,y\in\sub_2(T\setminus\{n+1\})$
  determine the lines of ${\goth N}'$.
\end{proof}


\section{Particular case: ${\goth N}$ is a generalized Desargues configuration}
\label{subsec:grasaxis}

In the class of skew perspectives one type of them seems ``most similar to the
classical geometrical perspective'': when the perspective axis is a generalized
Desargues configuration i.e. when ${\goth N} = \GrasSpace(n,2)$
(cf. \cite{doliwa1}, \cite{doliwa2}).
So, in this subsection we set 
\begin{ctext}
  ${\goth M} = \perspace(n,\sigma,{\GrasSpace(n,2)})$, $\sigma\in S_I$, $n \geq 4$.
\end{ctext}

\begin{prop}
  Either ${\goth M} = \GrasSpace(n+2,2) = \perspace(n,\id)$ and then each point
  of $\goth M$ can be chosen as a center of a skew perspective, or
  $\goth M$ does not contain any point $q\neq p$ such that 
  ${\goth M} \cong \perspace(n,\sigma',{\goth B}) =: {\goth M}'$ 
  for a suitable configuration ${\goth B}$,
  such that $q$ is the perspective center in ${\goth M}'$.
\end{prop}
\begin{proof}
  Assume that $\sigma \neq \id_I$.
  Suppose that such a point $q$ exists, then -- comp. 
  \ref{prop:movecenter} and \ref{lem:nextgraf} --
  there is $i_0 \in I$ such that $\sigma(i_0) = i_0$.
  Moreover, in view of \eqref{war:extraskos}, there is a permutation $\tau$ such that
  $c_{i_0,\tau(i)} \oplus c_{i_0,\tau(j)} = c_{i,j}$ 
  for all $i,j \in I$, $i,j\neq i_0$. 
  On the other hand, in $\GrasSpace(I,2)$ we have  
  $c_{i_0,\tau(i)} \oplus c_{i_0,\tau(j)} = c_{\tau(i),\tau(j)}$ 
  for all $i,j$ as above. 
  This, finally, gives
  $\{ i,j \} = \{ \tau(i),\tau(j) \}$, from which we deduce $\tau = \id$ and then
  ${\goth M}' \cong \GrasSpace(n+2,2)$.
\end{proof}
\begin{cor}\label{cor:iso2}
  Let $S_I \ni \sigma_1\neq\id_I$.
  If $f$ is an isomorphism between 
  ${\perspace(n,\sigma_1,{\GrasSpace(n,2)})}$ and ${\perspace(n,\sigma_2,{\GrasSpace(n,2)})}$
  then $f(p) = p$ and $\sigma_2\neq \id_I$. 
  Moreover, 
  $f$ is determined by a permutation $\varphi\in S_I$ 
  (comp. \eqref{propiso1:typ1}, \eqref{propiso1:typ2})
  so as
  either $f$ fixes $A$ and $B$ and then 
  $\sigma_2 = \varphi\circ\sigma_1\circ \varphi^{-1} = \sigma_1^{\varphi}$,
  or $f$ interchanges $A$ and $B$ and $\sigma_2^{-1} = \sigma_1^{\varphi}$
  (see Prop. \ref{prop:iso1}).
\end{cor}

Let us recall a few facts from the folklore of group theory.
Let $\sigma\in S_I$, then $\sigma$ has a unique (up to an order) decomposition
$\sigma = \sigma_1\circ\ldots\circ\sigma_k$ where $\sigma_1,\ldots,\sigma_k$ are pairwise
disjoint cycles. Let $x_i$ be the length of $\sigma_i$, then $n = \sum_{i=1}^k x_i$.
Without loss of generality we can assume that $x_1\leq \ldots \leq x_k$ and we can 
set $C(\sigma) := (x_1,\ldots,x_k)$. So, $C(\sigma)$ is 
an unordered partition of the integer
$n$ into $k$ components (see e.g. \cite[Ch. 4]{hall}, \cite{bona}).
The following is known:
\begin{fact}\label{fct:conjug}
  $\sigma_1$ and $\sigma_2$ are conjugate in $S_I$ 
  (i.e. $\sigma_2 = \varphi\circ\sigma_1\circ\varphi^{-1} = \sigma_1^\varphi$ 
  for a $\varphi\in S_I$), 
  iff $C(\sigma_1) = C(\sigma_2)$.
  \par
  In particular, $\sigma$ and $\sigma^{-1}$ are conjugate for every $\sigma\in S_I$.
  \par
  Permutations $\sigma$ and $\id_I$ are conjugate iff $\sigma = \id_I$.
\end{fact}
As an immediate consequence of \ref{fct:conjug} and \ref{cor:iso2}
we obtain
\begin{prop}\label{prop:class-grasaxis}
  Let $\sigma_1,\sigma_2\in S_I$.
  $\perspace(n,\sigma_1,{\GrasSpace(n,2)}) \cong \perspace(n,\sigma_2,{\GrasSpace(n,2)})$
  iff $\sigma_1$ and $\sigma_2$ are conjugate.
  \par
  Consequently, there are $P(n) = \sum_{k=1}^n P(n,k)$ types of the skew perspectives
  whose axial configurations are the generalized Desargues configuration, where
  $P(n,k)$ is the number of unordered partitions of $n$ into $k$ components.
\end{prop}


\section{A few examples and counterexamples: some $\konftyp(15,4,20,3)$-configurations}
\label{subsec:konter}

In this Section we discuss some $\konftyp(15,4,20,3)$-configurations which appear to be
skew perspectives. Some of them were (up to an isomorphism) defined elsewhere,
they fall into some other classes of configurations.
Then we use the notation of the papers where `origins' can be found without definite 
explanation. But original definitions are useless in this place 
(sometimes we briefly quote the idea of a respective definition):
we merely want to show what `name' has the structure in that other papers.
No general important result follows from investigations of this Section;
the reader will stay more familiar with technical apparatus used in our theory
and with some fundamental examples of (really `skew') perspectives.

\begin{exm}\label{exm:0}
  Let $n=2k$, $I = I_{2k}$, and $\sigma = (1,2)(3,4)\ldots(2k-1,2k)$, \quad or 
  $n=2k+1$, $I = I_{2k} \cup \{0\}$, and $\sigma = (0)(1,2)(3,4)\ldots(2k-1,2k)$,
  for an integer $k\geq 2$. So, $\sigma$ is, in fact, a family of disjoint transpositions.
  The following is a direct consequence of \cite[Repr. 2.4]{skewgras}
  \begin{fact*}
    $\perspace(n,\sigma,{\GrasSpace(n,2)})$ is the combinatorial quasi Grassmannian 
    $\vergras_{n}$ of \cite{skewgras}.
  \end{fact*}
  In accordance with our theory developed in Subsection \ref{subsec:grasaxis},
  $\vergras_{2k}$ has exactly two $K_{2k+1}$ subgraphs and 
  $\vergras_{2k+1}$ has three $K_{2k+2}$-subgraphs 
  (see also \cite[Cor. 4.4]{klik:binom}).
\myend
\end{exm}
In particular, $\vergras_{4}$ is a $\konftyp(15,4,20,3)$-configuration.

\bigskip

All the $\konftyp(15,4,20,3)$-configurations with at least three free $K_5$-subgraphs
inside were listed in \cite[Classif. 2.8]{STP3K5}.
In particular, each of them is a binomial configuration which contains two maximal
complete subgraphs so, it is a perspective of two $K_5$ with an additional free $K_5$.
Let us analyse some, concrete, examples, which appear in accordance with
\ref{lem:nextgraf}.

Let ${\goth M} = \perspace(4,\sigma,{\GrasSpace(I_4,2)})$; suppose that 
$\Fix(\sigma)\neq \emptyset,I_4$ for a $\sigma\in S_{I_4}$.
\begin{exm}\label{exm:1}
  $\sigma = (1)(2,3,4)$.
  Then $\goth M$ coincides with the configuration defined in \cite[Classif. 2.8(ii)]{STP3K5}.
  To see this it suffices to represent it in the form of a system of triangle
  perspectives in accordance with Figure \ref{fig:exm1}.
  \begin{figure}
  \begin{center}
  \begin{minipage}[m]{0.6\textwidth}
    \xymatrix{%
    {\Delta_1:}
    &
    {c_{1,2}}\ar@{-}[dr]\ar@(dr,ur)@{-}[dd]
    &
    {c_{1,3}}\ar@{-}[dr]\ar@(dr,ur)@{-}[dd]
    &
    {c_{1,4}}\ar@{-}[dll]\ar@(dr,ur)@{-}[dd]
    \\
    {\Delta_2:}
    &
    {b_{3}}\ar@{-}[dr]
    &
    {b_{4}}\ar@{-}[dr]
    &
    {b_{2}}\ar@{-}[dll]
    \\
    {\Delta_3:}
    &
    {a_2}
    &
    {a_3}
    &
    {a_4}
    }
  \end{minipage}
  \end{center}
  \caption{The diagram of the line $\{ c_{2,3}, c_{2,4}, c_{3,4} \}$
  in $\perspace(4,{(1)(2,3,4)},{\GrasSpace(I_4,2)})$. \newline
  \strut\quad\quad 
  $c_{2,3} \in \overline{c_{1,2},c_{1,3}},\; \overline{b_3,b_4},\; \overline{a_2,a_3}$,
  $c_{3,4}\in \overline{c_{1,3},c_{1,4}},\; \overline{b_4,b_2},\; \overline{a_3,a_4}$,
  $c_{2,4} \in \overline{c_{1,2},c_{1,4}},\; \overline{b_3,b_2},\; \overline{a_2,a_4}$.
  \newline\strut\quad\quad
  $b_1$ is the centre of $\Delta_1$ and $\Delta_2$,
  $p$ is the centre of $\Delta_2$ and $\Delta_3$,
  and
  $a_1$ is the centre of $\Delta_1$ and $\Delta_3$
  (lines in the diagram join points which correspond each to other under respective
  perspective).%
\myend} 
  \label{fig:exm1}
  \end{figure}
\myend
\end{exm}
\begin{exm}\label{exm:2}
  $\sigma = (1)(2)(3,4)$. Then $\goth M$ coincides with the configuration defined in 
  \cite[Rem. 2.10(iii)]{STP3K5} -- cf. Figure \ref{fig:exm2}.
  Consequently, $\goth M$ is isomorphic to the so called {\em multi veblen} configuration
    $\xwlepp({I_4},{p},{L_{4}},{},{\GrasSpace(I_4,2)})$, 
  where $L_{4}$ is a linear graph on $I_4$.
  \begin{figure}
  \begin{center}
  \begin{minipage}[m]{0.6\textwidth}
\xymatrix{%
  {\Delta_1:}
  &
  {c_{1,2}}\ar@{-}[d]\ar@(dr,ur)@{-}[dd]
  &
  {c_{1,3}}\ar@{-}[d]\ar@(dr,ur)@{-}[dd]
  &
  {c_{1,4}}\ar@{-}[d]\ar@(dr,ur)@{-}[dd]
  \\
  {\Delta_2:}
  &
  {a_{2}}\ar@{-}[d]
  &
  {a_{3}}\ar@{-}[dr]
  &
  {a_{4}}\ar@{-}[dl]
  \\
  {\Delta_3:}
  &
  {b_2}
  &
  {b_4}
  &
  {b_3}
}
  \end{minipage}
  \end{center}
  \caption{The diagram of the line $\{ c_{2,3},c_{3,4},c_{2,4} \}$
  in $\perspace(4,{(1)(2)(3,4)},{\GrasSpace(I_4,2)})$. 
  \newline\strut\quad\quad
  $c_{2,3}\in\overline{c_{1,2},c_{1,3}},\;\overline{a_2,a_3},\;\overline{b_2,b_4}$,
  $c_{3,4}\in\overline{c_{1,3},c_{1,4}},\;\overline{a_3,a_4},\;\overline{b_4,b_3}$, and
  $c_{2,4}\in\overline{c_{1,2},c_{1,4}},\;\overline{a_2,a_4},\;\overline{b_2,b_3}$. 
  \newline\strut\quad\quad
  $a_1$ is the centre of $\Delta_1$ and $\Delta_2$,
  $p$ is the centre of $\Delta_2$ and $\Delta_3$, and
  $b_1$ is the centre of $\Delta_1$ and $\Delta_3$.%
  \myend}
  \label{fig:exm2}
  \end{figure}
\myend
\end{exm}
\begin{exm}\label{exm:3}
  Let  ${\goth M} = \xwlepp({I_4},{p},{L_{4}},{},{\GrasSpace(I_4,2)})$.
  It is known that 
      $\xwlepp({I_4},{p},{L_{4}},{},{\GrasSpace(I_4,2)}) \cong
      \xwlepp({I_4},{p},{K_{4}\setminus \{\{ 2,3 \}\}},{},{\GrasSpace(I_4,2)})$ 
  (cf. \cite[Thm. 4]{pascvebl}).
  Without coming into details let us quote 
  (after \cite[Constr. 4]{pascvebl}, 
  compare with Construction \ref{def:pers})
  that in an arbitrary multiveblen configuration 
  $\xwlepp({I},{p},{{\cal P}},{},{{\goth N}})$, 
  we have
  a centre $p$, the lines through $p$ with the points $a_i, b_i,\; i\in I$ as in
  ${\cal L}_p$,
  and a graph $\cal P$ defined on $I$
  which determines whether 
  $c_{i,j} = a_i \oplus a_j = b_i \oplus b_j$ ($\{ i,j \}\in{\cal P}$) or
  $c_{i,j} = a_i \oplus b_j = b_i \oplus a_j$ ($\{ i,j \}\notin{\cal P}$).
  Then the axis $\goth N$ is used as in the definition of $\perspace(n,\sigma,{\goth N})$
  to get ${\cal L}_C$.
  \par
  Let us quote after \cite[Cor. 2.13]{klik:binom} the following characterization,
  which will be needed in the sequel
  \par\noindent
  \begin{quotation}\noindent
  \refstepcounter{equation}\label{char:MV}
  { \em A $\binkonf(n,0)$-configuration is a multiveblen configuration with the axis $\GrasSpace(n-2,2)$ iff
  it contains at least $n-2$ free $K_{n-1}$-subgraphs.
  }\hfill\eqref{char:MV}
  \end{quotation}
  \par
  $\goth M$ can be represented as a perspective of two graphs
  $G_1 = \{ a_1,c_{1,2},c_{1,3},b_1 \}$ and $G_2 = \{ a_4,c_{2,4},c_{3,4},b_4 \}$
  with centre $q = c_{ 1,4 }$.
  \begin{fact*}
  ${\goth M} \cong \perspace(4,\id,{\goth N})$,
  where ${\goth N}\cong\VeblSpace(2)$ is the Veblen configuration with the lines
  \begin{ctext}
    $\{ 
    \{ e_{1,4}, e_{1,2}, e_{2,4} \}, 
    \{ e_{1,4}, e_{1,3}, e_{3,4} \}, 
    \{ e_{1,2}, e_{2,3}, e_{3,4} \},
    \{ e_{1,3}, e_{2,3}, e_{2,4} \}
    \}$,
  \end{ctext}
  $e_{i,j} = x_i\oplus x_j$, and $x_i$ are the vertices of $G_1$.
  \myend
  \end{fact*}
\end{exm}
Gathering together \ref{exm:2} and \ref{exm:3} we see that
\begin{ctext}
  $\perspace(4,\id,{\VeblSpace(2)}) \cong \perspace(4,{(1)(2)(3,4)},{\GrasSpace(I_4,2)})$
\end{ctext}
so, a skew perspective does not determine, geometrically, its centre and a 
labelling of the points in axial configuration.
\begin{exm}\label{exm:8}
  Let $\GrasSpacex(I_4,2)$ be the Veblen configuration whose lines are the $\varkappa$-images (see \ref{exm:7})
  of the lines of $\GrasSpace(I_4,2)$.
  Then, for every graph $\cal P$ defined on $I_4$ the structure
    ${\goth M} = \xwlepp({I_4},{p},{\cal P},{},{\GrasSpacex(I_4,2)})$
  contains four $K_5$-graphs: $G_i = \{ a_i,b_i \} \cup \{ c_{i,j}\colon j \in I_4 \setminus \{ i \} \}$
  with $i \in I_4$. However, no one of the $G_i$ is freely contained in $\goth M$
  and one can directly verify that $\goth M$ cannot be presented as a $\konftyp(15,4,20,3)$-perspective.
\myend
\end{exm}
\begin{exm}\label{exm:4}
  Let 
    ${\goth M} = \xwlepp({I_4},{p},{N_{4}},{},{\GrasSpace(I_4,2)})$
  (cf. \cite[Rem. 2.10(ii)]{STP3K5}, \cite[Constr.. 2]{pascvebl}),
  where $N_4$ is the empty graph on $4$ vertices.
  The structure $\goth M$ freely contains four $K_5$-subgraphs and it is homogeneous: 
  any two points in $C$ can be interchanged by an automorphism of $\goth M$.
  Let us represent $\goth M$ in the form $\perspace(4,\sigma,{\goth N})$ with 
  the centre
  $q = c_{1,2}$
  chosen as an example.
  Then the perspective graphs are 
  $G_1 = \{ a_1,b_1,c_{1,3},c_{1,4} \}$ and $G_2 = \{ b_2,a_2,c_{2,3},c_{2,4} \}$.
  We find then the following representation.
  \begin{fact*}
  ${\goth M} \cong \perspace(4,{(1,2)(3)(4)},{\goth N})$, where
  ${\goth N} \cong \VeblSpace(2)$
  is the Veblen configuration with the lines
  \begin{ctext}
    $\{ 
    \{ e_{1,2}, e_{1,3}, e_{2,3} \}, 
    \{ e_{1,2}, e_{1,4}, e_{2,4} \}, 
    \{ e_{1,3}, e_{2,4}, e_{3,4} \}, 
    \{ e_{1,4}, e_{2,3}, e_{3,4} \}
    \}$,
  \end{ctext}
  $e_{i,j}$ are defined as in \ref{exm:3}.
  \myend
  \end{fact*}
\end{exm}
\begin{exm}\label{exm:6}
  Examples \ref{exm:3} and \ref{exm:4} both can be generalized with the following
  computation. Let 
    ${\goth M} = \xwlepp({X},{p},{{\cal P}},{},{\GrasSpace(X,2)})$
  where $\cal P$ is a graph defined on $X$, $|X| = n$.
  Consider two complete free subgraphs $G_1,G_2$ of $\goth M$ intersecting in a point
  $q = c_{i,j}$. Without loss of generality we can assume that $i = 1,\, j = 2$
  and $X = \{ 1,\ldots, n \}$. Set $I_0 = \{ 3,\ldots,n \}$.
  Then
  \begin{multline}\label{post1}
   G_1 = \{ x_1 = a_1, x_2 = b_1, x_j = c_{1,j},\; j\in I_0 \} 
   \text{ and }
   \\
   G_2 = \{ y_1 = a_2, y_2 = b_2, y_j = c_{2,j},\; j\in I_0 \}. 
  \end{multline}
  Define \centerline{%
  \begin{math}  e_{i,j} = x_i \oplus x_j \text{ for } \{i,j\} \in \sub_2(X).\end{math}}
  Then we have
  \begin{equation}\label{nowaos1}
    e_{1,2} = p = y_1\oplus y_2,\quad
    e_{i,j} := c_{i,j} = y_i \oplus y_j 
    \text{ for all } \{ i,j \}\in \sub_2(I_0).
  \end{equation}
  Let us consider the two following cases:
  \begin{enumerate}[(A)] 
  \item\label{tak} $\{ 1,2 \} \in {\cal P}$ \item\label{nie} $\{ 1,2 \} \notin {\cal P}$.
  \end{enumerate}
  
  \par \strut\quad Assume \eqref{tak}.
  One can easily compute that $q = x_i \oplus y_i$ for $i\in I$.
  Moreover, we compute for $j \geq 3$ as follows:
  \par\noindent\centerline{
  $e_{1,j} = \left\{  \begin{array}{ll} 
               a_j & \text{ when }\{ 1,j \} \in {\cal P} \\
	       b_j & \text{ when }\{ 1,j \} \notin {\cal P}
	       \end{array} \right.$
  and
  $e_{2,j} = \left\{  \begin{array}{ll} 
               b_j & \text{ when }\{ 1,j \} \in {\cal P} \\
	       a_j & \text{ when }\{ 1,j \} \notin {\cal P}
	       \end{array} \right.$.}
  Analogously, we compute
  \par\noindent\centerline{
  $y_1\oplus y_j = \left\{  \begin{array}{ll} 
               a_j & \text{ when }\{ 2,j \} \in {\cal P} \\
	       b_j & \text{ when }\{ 2,j \} \notin {\cal P}
	       \end{array} \right.$
  and
  $y_2\oplus y_j = \left\{  \begin{array}{ll} 
               b_j & \text{ when }\{ 2,j \} \in {\cal P} \\
	       a_j & \text{ when }\{ 2,j \} \notin {\cal P}
	       \end{array} \right.$.}
  The formulas above and the formula \eqref{nowaos1} determine the skew:
  %
  \par\medskip\noindent\refstepcounter{equation}
  \label{skos1}
  \begin{math}
  \strut\hfill
  \sigma(\{ i,j \}) = \{ i,j\} 
    \text{ for } \{ i,j \} \in \sub_2(I_0) \cup \{\{ 1,2 \}\},
    \text{ let } j\geq 3:
   \hfill\strut \\
    \strut\quad\quad\sigma\colon \{ 1,j \} \mapsto \{ 2,j \} \mapsto \{ 1,j \}
    \\
    \strut \hfill \text{ when } \{ 1,j \}\in{\cal P},\{2,j\} \in {\cal P} \text{ or }
    \{ 1,j \}\notin{\cal P},\{2,j\} \notin {\cal P},
    \\
    \strut\quad\quad\sigma\colon \{ 1,j \} \mapsto \{ 1,j \}, 
    \sigma\colon \{ 2,j \} \mapsto \{ 2,j \}
    \\
    \strut\hfill \text{ when } \{ 1,j \},\{2,j\} \in {\cal P} \text{ or }
    \{ 1,j \},\{2,j\} \notin {\cal P}.
\end{math}\quad\eqref{skos1}
%
  \par\medskip\noindent
  Finally, let ${\cal P}_0$ be the restriction of $\cal P$ to $\sub_2(I_0)$.
  We conclude with
  \begin{facto}{\ref{exm:6}}
    In case \eqref{tak}, ${\goth M} \cong \perspace(n,\sigma,{\goth N})$,
    where ${\goth N} = \xwlepp({I_0},{p},{{\cal P}_0},{},{\GrasSpace(I_0,2)})$
    and $\sigma$ is defined by \eqref{skos1}.
  \end{facto}
  \par\noindent\strut\quad
  Now, let us pass to the case \eqref{nie}.
  In this case we only slightly renumber the elements of $G_1$ and $G_2$
  (cf. \eqref{post1}):
  \begin{multline}\label{post2}
   G_1 = \{ x_1 = a_1, x_2 = b_1, x_j = c_{1,j},\; j\in I_0 \} 
   \text{ and }
   \\
   G_2 = \{ y_1 = b_2, y_2 = a_2, y_j = c_{2,j},\; j\in I_0 \}. 
  \end{multline}
  Clearly, $e_{i,j}$ take values as in \eqref{tak}.
  Differences appear when we compute for $j\geq 3$
  \par\noindent\centerline{
  $y_1\oplus y_j = \left\{  \begin{array}{ll} 
               a_j & \text{ when }\{ 2,j \} \notin {\cal P} \\
	       b_j & \text{ when }\{ 2,j \} \in {\cal P}
	       \end{array} \right.$
  and
  $y_2\oplus y_j = \left\{  \begin{array}{ll} 
               b_j & \text{ when }\{ 2,j \} \notin {\cal P} \\
	       a_j & \text{ when }\{ 2,j \} \in {\cal P}
	       \end{array} \right.$.}
  Now, the skew is determined by the following conditions:
  \par\medskip\noindent\refstepcounter{equation}
  \label{skos2}
  \begin{math}
  \strut\hfill
    \sigma(\{ i,j \}) = \{ i,j\} 
    \text{ for } \{ i,j \} \in \sub_2(I_0) \cup \{\{ 1,2 \}\}, 
    \text{ let } j\geq 3:
    \hfill\strut \\
    \strut\quad\quad \sigma\colon \{ 1,j \} \mapsto \{ 2,j \} \mapsto \{ 1,j \}
    \\
    \strut\hfill \text{ when } \{ 1,j \},\{2,j\} \in {\cal P} \text{ or }
    \{ 1,j \},\{2,j\} \notin {\cal P},
    \\
    \strut\quad\quad
    \sigma\colon \{ 1,j \} \mapsto \{ 1,j \}, 
    \sigma\colon \{ 2,j \} \mapsto \{ 2,j \}
    \\
    \strut\hfill \text{ when } \{ 1,j \}\in{\cal P},\{2,j\} \in {\cal P} \text{ or }
    \{ 1,j \}\notin{\cal P},\{2,j\} \notin {\cal P}.
  \end{math}\quad\eqref{skos2}
  \par\medskip\noindent
  We conclude with
  \begin{facto}{\ref{exm:6}}
    In case \eqref{nie}, ${\goth M} \cong \perspace(n,\sigma,{\goth N})$,
    where ${\goth N} = \xwlepp({I_0},{p},{{\cal P}_0},{},{\GrasSpace(I_0,2)})$
    and $\sigma$ is defined by \eqref{skos2}.
  \end{facto}
  In particular, we obtain the following generalizations of \ref{exm:4}
  and a folklore.
  \begin{facto}{\ref{exm:6}}
    \begin{sentences}\itemsep-2pt
    \item
      $\xwlepp({X},{p},{N_X},{},{\GrasSpace(X,2)}) \cong 
       \perspace(n,\overline{\sigma},{\goth N})$,
      where 
      ${\goth N} = \xwlepp({I_0},{p},{N_{I_0}},{},{\GrasSpace(I_0,2)})$
      and $\sigma = {(1,2)(3)\ldots(n)}$.
    \item
      $\xwlepp({X},{p},{K_X},{},{\GrasSpace(X,2)}) \cong 
       \perspace(n,\overline{\sigma},{\goth N})$,
      where 
      ${\goth N} = \xwlepp({I_0},{p},{K_{I_0}},{},{\GrasSpace(I_0,2)})$
      and $\sigma = \id_X$.
  \myend
  \end{sentences}
  \end{facto}
\end{exm}
Finally, combining \ref{lem:nextgraf}, \eqref{char:MV},
and \ref{exm:6} we obtain the following.
\begin{prop}\label{prop:pers-mveb}
  Let $I = I_n$. Assume that $\goth M$ is not a generalized Desargues configuration.
  If a multiveblen configuration 
    ${\goth M} = \xwlepp({I},{p},{{\cal P}},{},{\GrasSpace(I,2)})$
  is isomorphic to 
    $\perspace(n,\sigma,{\goth N})$
  where $\sigma = \overline{\sigma_0}$, $\sigma_0\in S_I$
  and $\goth N$ is a binomial \PSTS\ defined on $\sub_2(I)$
  then, up to an isomorhism,
  $\sigma_0 = (1,2)(3)\ldots(n)$ and either
  $\{ 1,2 \}\in{\cal P}$, 
  $\{ 1,i \} \in {\cal P}$ iff $\{ 2,i \} \in {\cal P}$
  for all $j=3,\ldots,n$,
  or
  $\{ 1,2 \}\notin{\cal P}$, 
  $\{ 1,i \} \in {\cal P}$ iff $\{ 2,i \} \notin {\cal P}$
  for all $j=3,\ldots,n$,
  and 
  $\goth N$  is a multiveblen configuration determined by the graph obtained by 
  deleting from $\cal P$ the vertices $1$ and $2$.
\end{prop}
\begin{rem} 
  The two cases of $\{ 1,2 \} \{\in \text{ or } \notin\} {\cal P}$ above are, in fact,
  superflous. From \cite[Prop. 9]{pascvebl} we know that, up to an isomorphism
  we can always assume that $\{ 1,2 \}\in{\cal P}$.
\end{rem}
Consequently, \ref{exm:6} characterizes all the binomial configurations which
are simultaneously multiveblen configurations and skew perspectives 
preserving edge-concurrency.

\bigskip
Another example which is worth to consider is a combinatorial Veronesian $\VerSpace(3,k)$
of \cite{combver}.
This example shows, primarily, that not every "sensibly roughly presented" perspective 
$\perspace(n,\sigma,{\goth N})$ 
between complete graphs
has necessarily a `Desarguesian axis' nor its skew 
preserves the adjacency of edges of the graphs in question.
\begin{exm}\label{exm:5}
  Let us adopt the notation of \cite{combver}.
  Let $|X| = 3$, $X = \{ a,b,c \}$.
  Then the combinatorial Veronesian $\VerSpace(X,k) =: {\goth M}$ is 
  a $\binkonf(k,+2)$-configuration; 
  its point set is the set $\msub_k(X)$ of the $k$-element multisets with 
  elements in $X$
  and the lines have form $e X^s$, $e\in \msub_{k-s}(X)$.
  $\VerSpace(X,1)$ is a single line, $\VerSpace(X,2)$ is the Veblen configuration,
  and $\VerSpace(X,3)$ is the known Kantor configuration 
  (comp. \cite[Prop's. 2.2, 2.3]{combver}, \cite[Repr. 2.7]{klik:VC}). 
  Consequently, we assume $k > 3$.
  The following was noted in \cite[Fct. 4.1]{klik:binom}:
  \begin{ctext}
    The $K_{k+1}$ graphs freely contained in $\VerSpace(X,k)$ are the sets
    $X_{a,b} := \msub_k(\{ a,b \})$, 
    $X_{b,c} := \msub_k(\{ b,c \})$, and 
    $X_{c,a} := \msub_k(\{ c,a \})$.
  \end{ctext}
  In particular, $\goth M$ contains two complete subgraphs $X_{a,b}$, $X_{c,a}$, which cross
  each other in $p = a^k$. Let us present $\goth M$ as a perspective between
  these two graphs.
  Let us re-label the points of $\VerSpace(X,k)$:
  \begin{ctext}
    $c_i = b^i a^{k-i}$, $b_i = c^i a^{k-i}$, $i\in\{1,\ldots,k\} =: I$,\space
    $e_{i,j} = c_i \oplus c_j$, $\{ i,j \}\in \sub_2(I)$.
  \end{ctext}
  Clearly, $p \oplus c_i = b_i$ so, the map 
    $\big(c_i \mapsto b_i,\; i\in I\big)$
  is a point-perspective.
  Let us define the permutation $\sigma$ of $\sub_2(I)$ by the formula
  \begin{ctext}
  $\sigma(\{ i,j \}) = \{ j-1,j \}$ when $1\leq i < j \leq k$.
  \end{ctext}
  It is seen that $\sigma = \sigma^{-1}$.
  After routine computation we obtain  
  $b_i \oplus b_j = c_{\sigma(\{ i,j\})}$ whenever $i < j$; moreover, 
  in this representation the axial configuration consists of the points
in $bc \msub_{k-2}(X)$ so, it is isomorphic to $\VerSpace(X,k-2)$.
  Consequently, $\VerSpace(X,k) \cong \perspace(k,\sigma,{\VerSpace(X,k-2)})$.
%
It is seen that there is no permutation $\varphi\in S_I$ such 
that $\{ \varphi(i),\varphi(j) \} = \{ j-i,j \}$ for all $i < j$, 
unless $|I|=2 \not\geq 4$.
This can be summarized in the following
\begin{fact*}
  The binomial configuration $\VerSpace(3,k)$ with $k>3$ cannot be presented
  as a skew perspective, with the skew determined by a permutation or by
  the complementing in the set of indices. Though it represents a perspective of two 
  simplices.
\myend
\end{fact*}
\end{exm}

\section{Few remarks on projective realizability of skew perspectives}\label{ssec:pers2proj}

Our construction \ref{def:pers}, a generalization of a projective perspective, originates in studying
arrangements of points and lines of a (real) projective space.
So, the question whether (an which) skew perspectives can be realized in a Desarguesian 
projective space is quite natural. For ${10}_{3}$-configurations of the type
$\perspace(3,\sigma,{\GrasSpace(I_3,2)})$ the answer is affirmative (all three are realizable!) and is known for ages.
For structures $\perspace(4,\sigma,{\GrasSpace(I_4,2)})$, which are primarily investigated in this Section,
situation is more complex.
Let us begin with results easily derivable from known facts.

\begin{prop}\label{pers2proj:mveb}
  Let $\sigma\in S_{I_n}$ and $C(\sigma)$ be one of the following:
  $(1,\ldots,1)$, $(1,2,\ldots,2))$, $(2,\ldots,2)$.
  Then $\perspace(n,\sigma,{\GrasSpace(I_n,2)})$ can be realized in a real projective space.
\end{prop}
\begin{figure}
\begin{center}
\includegraphics[scale=0.5]{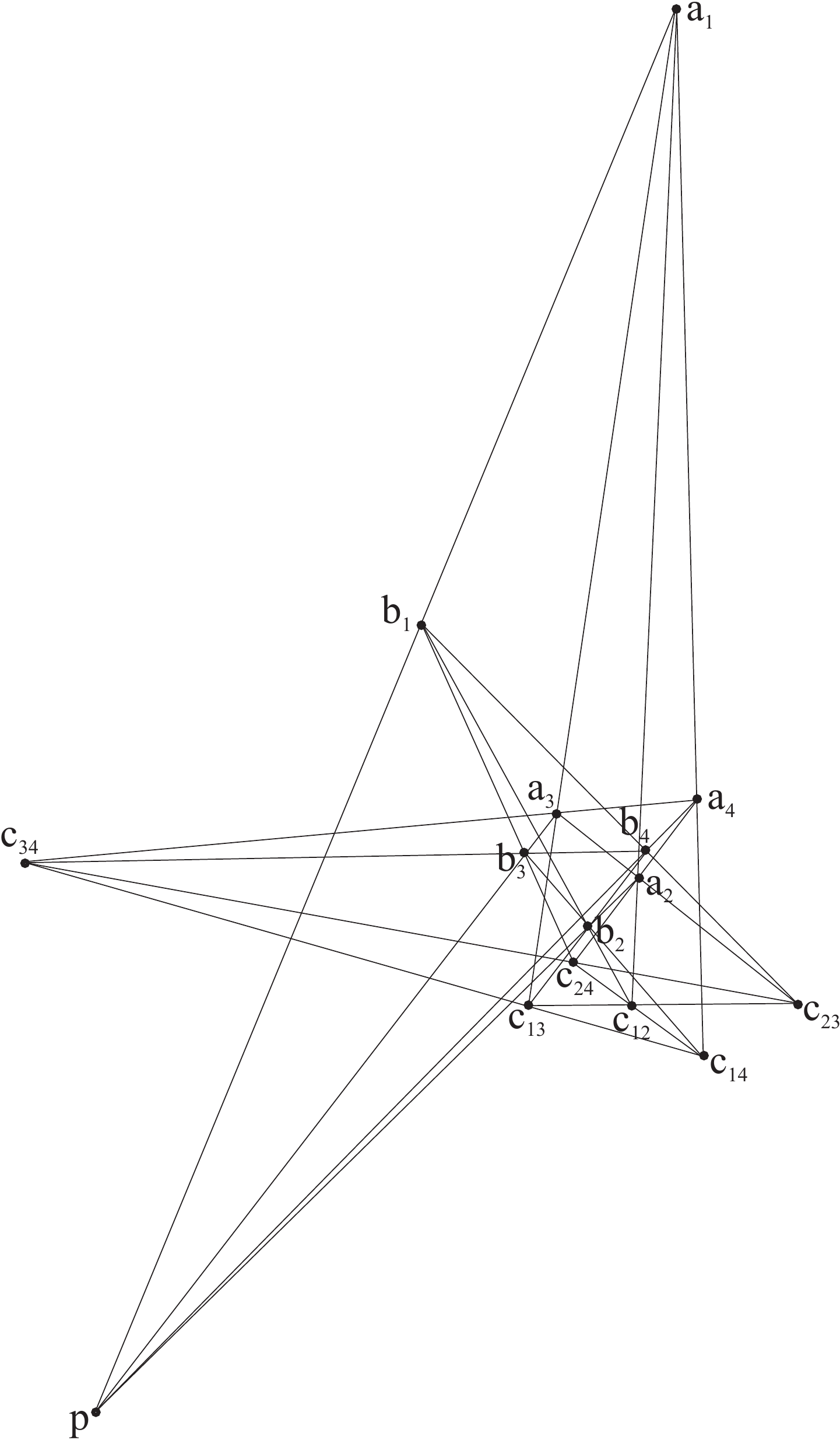}
\end{center}
\caption{The structure $\perspace(4,{(1,2)(3,4)},{\GrasSpace(I_4,2)}) = {\goth R}_4$, %
the smallest {\em not commonly known} example of the structures defined in \ref{pers2proj:mveb}.}
\label{fig:2xC2}
\end{figure}

\begin{proof}
  Write ${\goth M} = \perspace(n,\sigma,{\GrasSpace(I_n,2)})$
  Note that in the first case $\sigma = \id_{I_n}$, and $\goth M$ is the generalized Desargues configuration, 
  see \eqref{gras:pers}. 
  In the second and the third case $\sigma$ can be written in the form $(n)(1,2)(3,4)\ldots(n-2,n-1)$ and
  $(1,2)\ldots(n-1,n)$ resp. and then $\goth M$ is a combinatorial quasi Grassmannian, see Example \ref{exm:0}.
  In all these cases the claim follows from the results of
  \cite[Thm. 2.17]{mveb2proj} and \cite[Prop. 1.6 and Prop.'s 3.6-3.8]{skewgras}.
\end{proof}
We have also an evident lemma:
\begin{lem}\label{lem:ciach1}
  Let $\sigma\in S_I$, $J\subset I$, and $\sigma(J) = J$; set $\sigma_0 := \sigma\restriction_{J}$. 
  Then $\perspace(|J|,\sigma_0,{\GrasSpace(J,2)})$ is a subconfiguration of $\perspace(|I|,\sigma,{\GrasSpace(I,2)})$.
\end{lem}
\begin{prop}\label{pers2proj:L4}
  Let $\sigma\in S_{I_n}$. Assume that $C(\sigma)$ contains the sequence $(1,1,2)$ as its subsequence.
  Then $\perspace(n,\sigma,{\GrasSpace(I_n,2)})$ cannot be realized in any Desarguesian projective space.
\end{prop}
\begin{proof}
  Write ${\goth M} = \perspace(n,\sigma,{\GrasSpace(I_n,2)})$, 
  $\sigma_0 = (1)(2)(3,4)$, and 
  ${\goth M}_0 = \perspace(n,\sigma_0,{\GrasSpace(I_4,2)})$.
  Clearly, ${\goth M}_0$ is a subconfiguration of $\goth M$.
  From Example \ref{exm:2} and \cite[Prop. 2.3]{mveb2proj} we know that ${\goth M}_0$ cannot be realized in any
  Desarguesian projective space, which closes our proof.
\end{proof}
%
%

We say that a configuration $\goth M$ is {\em planar} if for any realization of $\goth M$ in a projective 
space $\goth P$ this realization lies on a plane of $\goth P$. Note that, anyway, even if $\goth M$
canot be realized in any Desarguesian projective space then it can be extended to a projective plane. 
So, in fact, in the definition above we can restrict ourselves to Desarguesian $\goth P$. And a configuration
nonrealizable in a Desarguesian projective space is, by definition, planar.

\begin{lem}\label{lem:p2p:plan1}
  Let $\sigma\in S_n$ be a cycle of length $n$, $n\geq 3$.
  The configuration $\perspace(n,\sigma,{\GrasSpace(I_n,2)})$ is planar. 
\end{lem}
\begin{proof}
  Consider a realization of $\perspace(n,\sigma,{\GrasSpace(I_n,2)})$ in a projective space $\goth P$.
  As it is commonly accepted, we do not distinguish a point of a configuration and its image under
  a realization in question.
  \par
  Let $A$ be the plane of $\goth P$ which contains $p,a_1,a_2$. 
  Then $b_2 = p\oplus a_2$ and $e_{1,2} = a_1\oplus a_2$ are on $A$.
  So, $b_3 = e_{1,2}\oplus b_2 \in A$ and then $a_3 = p \oplus b_3 \in A$. Inductively, we come to $a_i,b_i\in A$
  for all $i\in I_n$, which closes our proof.
\end{proof}
\begin{lem}\label{lem:p2p:plan2}
  Let $\sigma\in S_{I_n}$, 
  $\sigma(i_0) = i_0$, and
  $\sigma(i_1) = i_2 \neq i_1$
  for some $i_0, i_1, i_2 \in I_n$. Set $J:= I\setminus \{ i_0 \}$.
  If ${\goth M} = \perspace(n,\sigma,{\GrasSpace(I_n,2)})$ is embedded via $\gamma$
  into a Desarguesian projective space $\goth P$
  and the image under $\gamma$ of the 
  subconfiguration ${\goth N} = \perspace(n-1,\sigma\restriction{J},{\GrasSpace(J,2)})$ of $\goth M$
  lies on a plane $A$ of $\goth P$
  then the image of $\goth M$ under $\gamma$ lies on $A$.
  \par
  In particular, if $\goth N$ is planar then $\goth M$ is planar as well.
\end{lem}
\begin{proof}
  Suppose that $a_{i_0}\notin A$. Then the plane $B$ spanned in $\goth P$ by the points 
  $p, a_{i_0}, a_{i_1}$ is distinct from $A$ and it contains $b_{i_0}$.
  However, the lines $\LineOn(a_{i_0},a_{i_1})$ and  $\LineOn(b_{i_0},b_{i_2})$ intersect in $c_{i_0,i_1}\in B$,
  so $b_{i_2} \in B$. Consequently, $p,a_{i_1},b_{i_1},b_{i_2} \in A,B$ so, they are collinear
  and $\goth N$ degenerate. 
\end{proof}

As an direct consequence of \ref{lem:p2p:plan2} and \ref{lem:p2p:plan1} we obtain.
\begin{lem}\label{lem:p2p:plan3}
  Let   
  $C(\sigma) = (\underbrace{1,\ldots,1}_{(n-k)-\text{times}},k)$, $k \geq 3$.
  Then $\perspace(n,\sigma,{\GrasSpace(I_n,2)})$ is planar.
\end{lem}
%

Finally, with the help of the computer program Maple we can decide which of the remaining perspectives
$\perspace(4,\sigma,{\GrasSpace(I_4,2)})$ can be projectively realized. These cases are 
$C(\sigma) = (1,3)$ and $C(\sigma) = (4)$.

\begin{lem}\label{lem:4analyt}\setcounter{facto}{0}
  Let a system of points $p,a_i,b_i, \; 1\leq i\leq 4$ of the real projective plane $\goth P$ 
  be characterized by the following parametric equations.
  \begin{multline}\label{equ:embed}
  p = [1,0,0], 
  a_1 = [0,0,1], b_1 = [1,0,1], 
  a_2 = [1,\alpha_1,\alpha_2], b_2 = [1,\alpha_1 x,\alpha_2 x],
  \\
  a_3 = [0,1,0], b_3 = [1,1,0],
  a_4 = [1,\beta_1,\beta_2], b_4 = [1,\beta_1 y,\beta_2 y]. 
  \end{multline}
  \begin{sentences}
  \item\label{lem4analyt:casC4}
    If $\sigma = (1,2,3,4)$ then the above system of points yields in $\goth P$ a configuration isomorphic to
    $\perspace(4,\sigma,{\GrasSpace(I_4,2)})$ iff
    \begin{multline}\label{equ:embed1}
      \beta_1 = - \frac{-1 + \beta_2 y}{y},
      \alpha_1 = - \frac{1 - 2 \beta_2 y + \beta_2^2 y^2}{x y (\beta_2 -1)}, \\
      \alpha_2 = \frac{\beta_2 y (\beta_2^2 y^2 - 2 \beta_2 y + 1 - x y + x y \beta_2)}{%
      (\beta_2^2 y^2 - \beta_2 y - y + 1)}
    \end{multline}
    and a (terribly long) equation which assures that $c_{1,2}, c_{1,3}, c_{1,4}$ are not collinear holds.
    \begin{facto}{\ref{lem:4analyt}}\label{lem:pers2proj:C4}
      As an example we can quote that substituting 
      $\beta_2 := 2; \alpha_2 := -1, x := 2, y:= 2$ and taking into account \eqref{equ:embed1} 
      we do obtain a realization of $\perspace(4,\sigma,{\GrasSpace(I_4,2)})$.
      Consequently, \begin{quotation} $\perspace(4,{(1,2,3,4)},{\GrasSpace(I_4,2)})$ can be realized in the real
      projective plane. \end{quotation}
    \end{facto}
  \item\label{lem4analyt:casC3iC1}
    If $\sigma = (1,2,3)(4)$ then the above system of points yields in $\goth P$ a configuration isomorphic to
    $\perspace(4,\sigma,{\GrasSpace(I_4,2)})$ iff
    \begin{equation}\label{equ:embed2:1}
      \alpha_1 = - \frac{-1 + \alpha_2 x}{x},
    \end{equation}
    which guarantees that the points $p,a_i,b_i$, $i\leq 3$ yield the fez configuration
    $\perspace(3,{(1,2,3),\GrasSpace(I_3,2)})$, and
    \begin{equation}\label{equ:embed2:2}
      \alpha_2 = - \frac{\beta_2 (-1 + x)}{x \beta_1}, \;
      x = \frac{\beta_2^2 - \beta_2 \beta_1 + \beta_1^2}{\beta_2^2},
      \beta_1 \neq - \frac{\beta_2 y -1}{y}.
    \end{equation}
    The last relation in \eqref{equ:embed2:2} assures that $c_{1,2}, c_{1,3}, c_{1,4}$ are not collinear.
    \begin{facto}{\ref{lem:4analyt}}\label{lem:pers2proj:C3iC1}
      Substituting, concretely, 
      $\beta_1 := 5$, $\beta_2 := 2$, $y := 2$ and using \eqref{equ:embed2:1}, \eqref{equ:embed2:2}
      we arive to an example of concrete realization of
      $\perspace(4,\sigma,{\GrasSpace(I_4,2)})$. 
      Consequently \begin{quotation} $\perspace(4,{(1,2,3)(4)},{\GrasSpace(I_4,2)})$ can be realized in the real
      projective plane.  \end{quotation}
    \end{facto}
  \end{sentences}
\end{lem}
\begin{note}
  Due to the homogeneity of desarguesian planes the system \eqref{equ:embed} characterizes, in fact,
  an {\em arbitrary} system of points $a_i,b_i, \; i\leq 4$ that are point-perspective
  with the centre $p$.
\end{note}

\begin{figure}
\begin{center}
\includegraphics[scale=0.4]{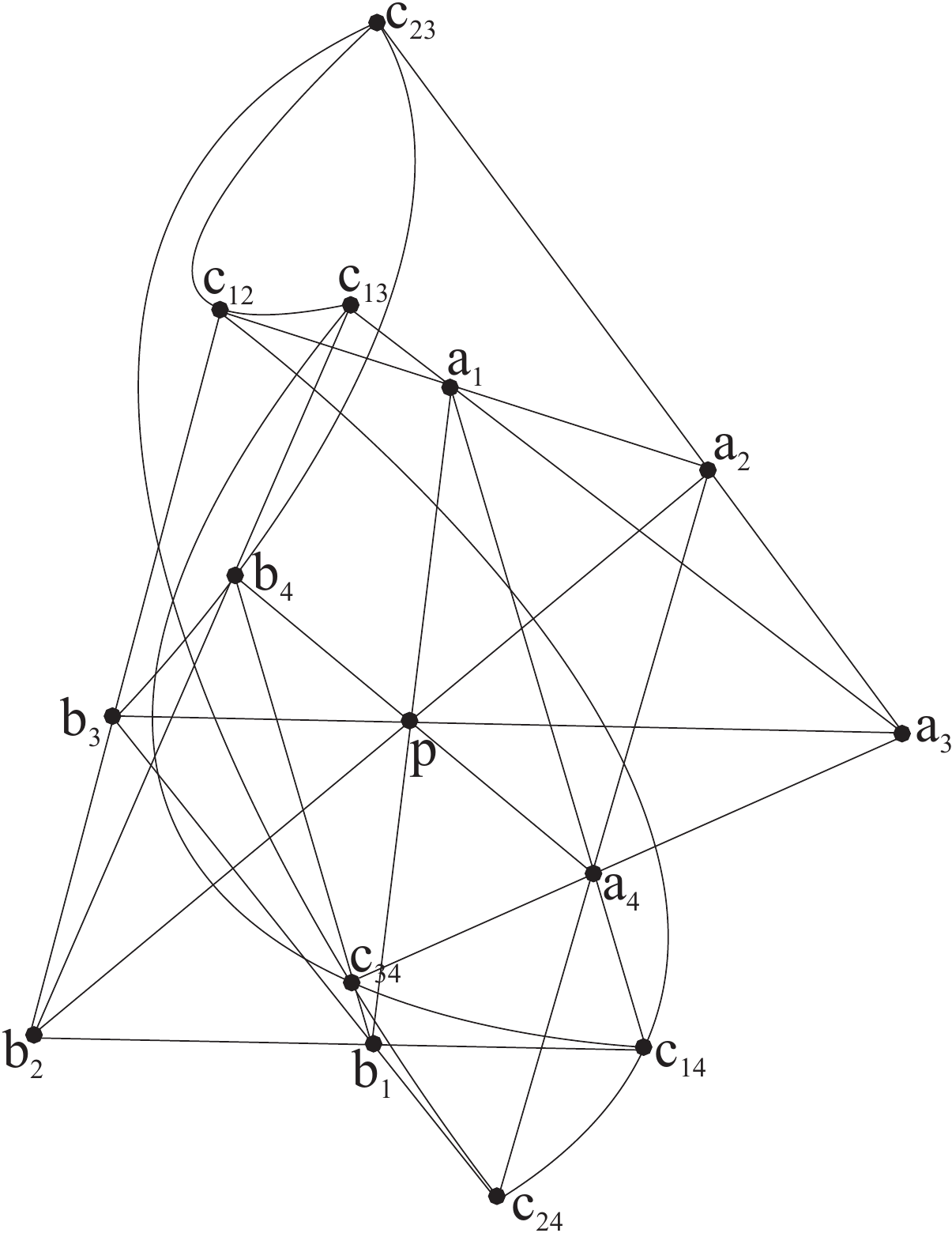}\quad
\includegraphics[scale=0.44]{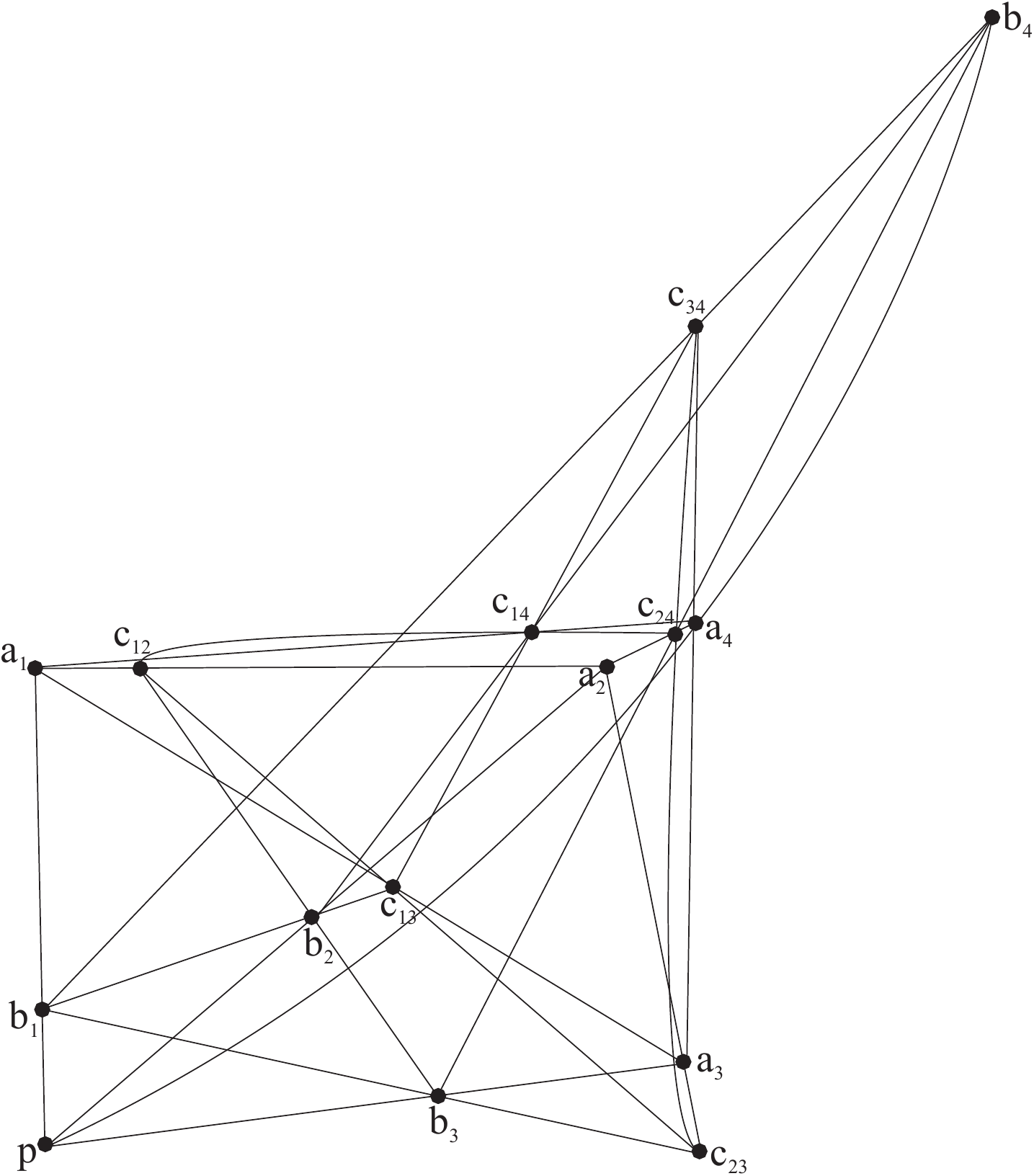}
\end{center}
\caption{The structures $\perspace(4,{(1,2,3,4)},{\GrasSpace(I_4,2)})$ (left) %
and $\perspace(4,{(1,2,3)(4)},{\GrasSpace(I_4,2)})$ (right), %
see \ref{lem:4analyt}.\ref{lem:pers2proj:C4} \space
and  \ref{lem:4analyt}.\ref{lem:pers2proj:C3iC1}. 
Schemas!: lines are dawn here as curved segments.}
\label{fig:persC4}
\end{figure}

As a somewhat tricky generalization of \ref{pers2proj:L4} let us note the following
\begin{fact}
  Let $C(\sigma)$ contain the sequence $(1,1,3)$ as its subsequence for a permutation $\sigma\in S_n$.
  Then ${\goth M} = \perspace(n,\sigma,{\GrasSpace(I_n,2)})$ cannot be realized in
  any Desarguesian projective space.
\end{fact}
\begin{proof}
  Let $\sigma = (1,2,3)(4)(5)$. 
  Suppose that ${\goth N} = \perspace(5,\sigma,{\GrasSpace(I_5,2)})$ can be realized in a Desarguesian 
  projective space; from \ref{lem:p2p:plan3}, $\goth N$ is realizable on a Desarguesian plane $A$.
  Without loss of generality we can assume that the points $p,a_i,b_i$ are defined by the system
  \eqref{equ:embed}, and 
  $a_5 = [1,\delta_1,\delta_2]$, $b_5 = [1,\delta_1 z, \delta_2 z]$.
  Since both systems of points: $p,a_1,a_2,a_3,a_4$ and $p,a_1,a_2,a_3,a_5$ yield 
  on $A$  
  (together with the respective $b_i$) 
  subconfigurations of $\goth N$ isomorphic to $\perspace(4,{(1,2,3)(4)},{\GrasSpace(I_4,2)})$,
  from \ref{lem:4analyt} we infer 
  $\frac{\beta_2}{\beta_1} = -\frac{x \alpha_2}{1 + x} = \frac{\delta_2}{\delta_1}$,
  which yields that $p,a_4,a_5$ are collinear, and this is impossible.
  Now the claim is evident, as $\goth M$ contains $\goth N$.
\end{proof}


\section{A few configurational properties: %
an analogue of the Desargues Axiom}\label{ssec:axioms}

Other group of problems which are commonly related to configurations similar to the Desargues
configuration are so called configurational axioms.
Let us briefly quote a formulation of the Desargues Axiom in the form which is suitable for our
purposes here:
\begin{quotation}\em
  Let $\cal A$ be a family of $10$ points in a (Desarguesian)
  projective space $\goth P$ such that after an identification $\gamma$
  of the points in $\cal A$ and the points of $\GrasSpace(I_5,2)$ $\gamma$ maps $9$ of the collinear triples
  of $\GrasSpace(I_5,2)$ onto triples collinear in $\goth P$ and no noncollinear triple is mapped onto a collinear one.
  Then the last, remaining, collinear triple in $\GrasSpace(I_5,2)$ is mapped by $\gamma$ onto a 
  collinear one. \hfill\desAx
\end{quotation}
We say that the Desargues configuration {\em closes} in Desarguesian projective spaces.
Clearly, such an elegant formulation of the Desargues axiom is possible because of the symmetries of the 
Desargues configuration. Analogous statement (with `$10$' and `$9$' replaced by suitable values `$\binom{n}{2}$'
and `$\binom{n}{3}-1$' is valid for generalized Desargues configuration $\GrasSpace(n,2)$ 
(comp. \cite[Prop. 1.9]{perspect}).
Nevertheless, one can prove that, in a sense, every (not too small) skew perspective associated with a permutation of
indices closes in every Desarguesian space.
\par
Let us begin with an evident observation.
\begin{lem}\label{lem:gras-krycie}
  Let $\sigma\in S_{I_n}$. Then (in the notation of \ref{def:pers}) each of two sets $A \cup C$ and $B \cup C$
  yields in $\perspace(n,\sigma,{\GrasSpace(I_n,2)})$ a subconfiguration that is a generalized Desargues configuration 
  of the type $\GrasSpace({n+1},2)$.
\end{lem}
From this we easily obtain the following form of ``configurational closeness'' of skew perspectives.
\begin{thm}\label{thm:pers-des}
  Let $\cal D$ be a set of $\binom{n+2}{2}$ points of a Desarguesian projective space $\goth P$ and let $\gamma$ be a
  bijection of $\cal D$ and
  the points of ${\goth M} = \perspace(n,\sigma,{\GrasSpace(I_n,2)})$, $\sigma\in S_{I_n}$, $n \geq 4$.
  Assume that
  \begin{sentences}\itemsep-2pt
  \item\label{ax:pers:1}
    $\gamma$ maps collinear triples of the form $p,a_i,b_i$ onto triples collinear in $\goth P$,
  \item\label{ax:pers:2}
    $gamma$ maps $\binom{n+2}{3} - n -1$ of the remaining collinear triples of $\goth M$ onto triples
    collinear in $\goth P$, and
  \item
    no noncollinear triple of points of $\goth M$ is mapped onto a collinear one.
  \end{sentences}
  Then the last triple of collinear points of $\goth M$ 
  (recall: $\goth M$ has $\binom{n+2}{3}$ triples of collinear points) 
  is mapped by $\gamma$ onto a collinear triple.
\end{thm}
\begin{proof}
  From assumptions, this `last' triple $L$ of collinear points has one of the following forms:
  \begin{equation}
    L = \{ a',a'',c \}, \quad  L = \{ b',b'',c \} \quad 
    \text{ or }\quad L = \{ c,c',c'' \}
  \end{equation}
  for $a',a'' \in A$, $b',b'' \in B$, $c,c'c''\in C$.
  In any case, by \ref{lem:gras-krycie} $L$ is contained in a 
  generalized Desargues configuration ${\cal G} \cong \GrasSpace(m,2)$ with $m\geq 5$, 
  contained in $\goth M$.
  From assumptions, $\gamma$ maps all the colinear triples of $\cal G$ except possibly $L$ onto collinear triples.
  So, $\gamma(L)$ is collinear as well.
\end{proof}
\begin{rem}
\begin{sentences}
\item
  {\em One cannot formulate \ref{thm:pers-des} as a full analogue of {\desAx}. 
  Namely, the conditions \ref{thm:pers-des}\eqref{ax:pers:1} must be placed in the assumptions.}
  Indeed,
  there is an embedding $\gamma$ of ${\goth M} = \perspace(4,{(1,2,3)(4)},{\GrasSpace(I_4,2)})$ into a real projective
  plane so as all the collinear triples of $\goth M$ except the triple $L = \{p, a_4,b_4  \}$ are mapped into 
  collinear triples, but $\gamma(L)$ is not collinear.
  Even a more impressive is the fact (a folklore, in fact), that there is an embedding of the fez configuration
  which preserves all the collinearities except $\{ p,a_3,b_3 \}$. 
\item
  It is a folklore, again, that {\em \ref{thm:pers-des} is not valid for $n=3$}; consider equation \eqref{equ:embed2:1}
  in \ref{lem:4analyt}, which is not a tautology on the real plane.\qed
\end{sentences}
\end{rem}

\fi 


\let\bibtem\bibitem

\noindent Author's address:\newline
\ifcdn{Kamil Maszkowski, }\fi
Ma{\l}gorzata Pra{\.z}mowska, Krzysztof Pra{\.z}mowski\\
Institute of Mathematics, University of Bia{\l}ystok\\
ul. Cio{\l}kowskiego 1M\\
15-245 Bia{\l}ystok, Poland\\
e-mail: \ifcdn{\ttfamily{kamil33221@o2.pl}, }\fi {\ttfamily malgpraz@math.uwb.edu.pl},
{\ttfamily krzypraz@math.uwb.edu.pl}


\begin{thebibliography}{9}\itemsep-2pt
\bibitem{doliwa1}
  {\sc A. Doliwa},
  {\it The affine Weil group symmetry of Desargues maps and the non-commutative
  Hirota-Miwa system},
  Phys. Lett. A {\bf 375} (2011), 1219--1224.
\bibitem{doliwa2}
  {\sc A. Doliwa},
  {\it Desargues maps and the Hirota-Miwa equation},
  Proc. R. Soc. A {\bf 466} (2010), 1177--1200.
\bibitem{perspect}
  {\sc M. Pra{\.z}mowska},
  {\it Multiple perspectives and generalizations of the Desargues configuration},
  Demonstratio Math. {\bf 39} (2006), no. 4, {887--906}.
\bibtem{pascvebl}
  {\sc M. Pra{\.z}mowska, K. Pra{\.z}mowski},
  {\it Some generalization of Desargues and Veronese configurations},
  Serdica Math. J. {\bf 32} (2006), no {2--3}, {185--208}.
\bibitem{klik:binom}
  {\sc M. Pra{\.z}mowska, K. Pra{\.z}mowski},
  {\it Binomial partial Steiner triple systems containing complete graphs},
  Graphs Combin. {\bf 32}(2016), no. 5, {2079--2092}.
\bibitem{klik:VC}
  {\sc K. Petelczyc, M. Pra{\.z}mowska},
  {\it ${10}_{3}$-configurations and projective realizability of multiplied configurations},
  Des. Codes Cryptogr. {\bf 51}, no. 1 (2009), {45--54}. 
\bibitem{klin}
  {\sc M. Ch. Klin, R. P{\"o}schel, K. Rosenbaum},
  {\it Angewandte Algebra f{\"u}r Mathematiker und Informatiker},
  VEB Deutcher Verlag der Wissenschaften, Berlin 1988
\bibitem{skewgras}
  {\sc M. Pra\.zmowska},
  {\it On some regular multi-{V}eblen configurations, the geometry of combinatorial 
  quasi {G}rassmannians},
  Demonstratio Math. {\bf 42}(2009), no.1 {2}, {387--402}.
\bibitem{combver}
  {\sc M. Pra{\.z}mowska, K. Pra{\.z}mowski},
  {\it Combinatorial Veronese structures, their geometry, and problems of %
  embeddability},
  Results Math. {\bf 51} (2008), 275--308.
\bibitem{STP3K5}
  {\sc K. Petelczyc, M. Pra{\.z}mowska},
  {\it A complete classification of the $(15_4 20_3)$-configurations with 
  at least three {$K_5$}-graphs},
  Discrete Math. {\bf 338} (2016), no {7}, {1243--1251}.
\bibitem{pasz}
  {\sc A. C. H. Ling, C. J. Colbourn, M. J. Granell, T. S. Griggs},
  {\it Construction techniques for anti-Pasch Steiner triple systems},
  Jour. London Math. Soc. {\bf 61} (2000), no. 3, 641--657.
\bibitem{coxdes}
  {H. S. M. Coxeter},
  {\it Desargues configurations and their collineation groups},
  Math. Proc. Camb. Phil. Soc. {\bf 78}(1975), 227--246.
\bibitem{projmono1}
  {\sc H. S. M. Coxeter},
  {\sl Introduction to Geometry},
  John Wiley, 1989.
\bibitem{projmono2}
  {\sc R. Hartshorne}, 
  {\sl Foundations of projective geometry}, 
  Lecture Notes, Harvard University, 1967.
\bibitem{projchain}
  {\sc H. Karzel, H.-J. Kroll},
  {\it Perspectivities in Circle Geometries},
  [in]
  {\sl Geometry -- von Staudt's point of view}, 
  {P. Plaumann, K. Strambach} (Eds), D. Reidel Publ. Co., 1981,
  pp. 51--100.
\ifcdn\relax\else
\bibitem{maszko}
  {\sc K. Maszkowski, M. Pra{\.z}mowska, K. Pra{\.z}mowski},
  {\it Configurations representing a skew perspective},
  mimeographed.
\fi
\bibitem{projectiv}
  {\sc G. Pickert},
  {\it Projectivites in Projective Planes},
  [in]
  {\sl Geometry -- von Staudt's point of view}, 
  {P. Plaumann, K. Strambach} (Eds), D. Reidel Publ. Co., 1981,
  pp. 1--50.
\bibitem{mveb2proj}
  {\sc M. M. Pra{\.z}mowska},
  {\it On the existence of projective embeddings of multiveblen configurations},
  {Bull. Belg. Math. Soc. Simon-Stevin},
  {\bf 17}, (2010), no {2}, {1--15}.
\bibitem{hall}
  {\sc M. Hall Jr.},
  {\sl Combinatorial Theory},
  J. Wiley 1986.
\bibitem{bona}
  {\sc G. E. Andrews},
  {\sl The theory of Partitions},
  Cambride Univ. Press, 1998.
\bibitem{hilbert}
  {\sc D. Hilbert, S. Cohn-Vossen},
  {\sl Geometry and the Imagination},
  AMS Chelsea Publishing, 1999.
\end{thebibliography}
\end{document}